\documentclass[a4paper,11pt]{amsart} 
\usepackage{amsmath,amssymb} 
\usepackage{mathtools}
\usepackage{amsthm}
\usepackage{mathrsfs}
\usepackage{graphics}
\usepackage[all]{xy} 
\usepackage{color}
\newcommand{\relmiddle}[1]{\mathrel{}\middle#1\mathrel{}}

\theoremstyle{plain}
\newtheorem{thm}{Theorem}

\newtheorem{prop}{Proposition}
\newtheorem*{main}{Main theorem}
\newtheorem*{introthm}{Theorem}

\theoremstyle{definition}

\newtheorem{eg}{Example}

\theoremstyle{remark}
\newtheorem{rem}{Remark}

\newcommand\R{\mathbb{R}}
\newcommand\C{\mathbb{C}}

\newcommand\pr{\mathbb{P}}
\newcommand\A{\mathbb{A}}


\usepackage[left=20mm,right=20mm,top=20mm,bottom=23mm,includefoot]{geometry}
\title[Autonomous limit of  4-dimensional Painlev\'e-type equations]{Autonomous limit of  4-dimensional Painlev\'e-type equations and degeneration of curves of genus two}  
\author{ 
Akane Nakamura
} 
\address{Department of Mathematics, Faculty of Science, Josai University\\
Keyakidai 1-1, Sakado, Sakado, 350-0295, Japan}
\email{a-naka@josai.ac.jp}



\usepackage{multicol}
\usepackage{paralist}

\usepackage{float}
\restylefloat{table}

\newif\ifpersonalnote \personalnotetrue

\newcommand{\personalnoteOFF}{%
  \ifpersonalnote\personalnotefalse}

\personalnoteOFF

\begin{document} 

\begin{abstract}  
	In recent studies, 4-dimensional analogs of the Painlev\'e equations were listed and there are 40 types.
	The aim of the present paper  
	is to geometrically characterize these 40 Painlev\'e-type equations.
	For this purpose, we study the autonomous  limit of
	these equations
	and 
	degeneration of their spectral curves. 
	The spectral curves are 2-parameter families of genus two curves and their 
	generic degeneration are one of the types classified by  Namikawa and Ueno.
	Liu's algorithm enables us to find the
	degeneration types  of the spectral curves for
	our 40 types of integrable systems.  
	This result is analogous to the following fact;
	the families of the spectral curves
	of the autonomous 2-dimensional Painlev\'e equations  $P_{\rm I}$,\ $P_{\rm II}$,\ $P_{\rm IV}$,\ $P_{\rm III}^{D_8}$,\ $P_{\rm III}^{D_7}$, \ $P_{\rm III}^{D_6}$,\ $P_{V}$ and $P_{\rm VI}$ define
	elliptic surfaces with the singular fiber at $H=\infty$ of the Dynkin types 
	$E_8^{(1)}$,\ $E_7^{(1)}$,\ $E_6^{(1)}$,\ $D_8^{(1)}$,\ $D_7^{(1)}$,\ $D_6^{(1)}$,\ $D_5^{(1)}$ and $D_4^{(1)}$, respectively.
\end{abstract}

\maketitle 
\numberwithin{equation}{section}

\section{Introduction}\label{intro}
The Painlev\'e equations are 8 types of nonlinear second-order ordinary differential equations with the Painlev\'e property\footnote{A differential equation is said to have the Painlev\'e property if 
its general solution has no critical singularities that depend on the initial values.
} which are not solvable by elementary functions.
The Painlev\'e equations have various interesting features;  they can be derived from isomonodromic deformation of certain linear equations, they are linked by degeneration process, they have affine Weyl group symmetries, they can be derived from reductions of soliton equations, and they  are equivalent to nonautonomous Hamiltonian systems.
Furthermore, their autonomous limits are integrable systems solvable by elliptic functions.
We call such systems the autonomous Painlev\'e equations.

Various generalization of the Painlev\'e equations have been proposed by focusing on one of the features of the Painlev\'e equations.
The main two directions of generalizations are  higher-dimensional analogs and difference analogs.
We treat the 4-dimensional analogs in this paper.
The eight types of the Painlev\'e equations are called the 2-dimensional Painlev\'e equations in this context, and the higher-dimensional analogs are $2n$-dimensional Painlev\'e-type equations ($n=1,2,\dots$).

Recently, classification theory of linear equations up to Katz's operations have been developed.
Since these operations of linear equations leave isomonodromic deformation equations invariant~\cite{MR2363425}, we can make use of classification theory of linear equations to  classify the Painlev\'e-type equations\footnote{In this paper we use the term the Painlev\'e-type equations synonymously with  isomonodromic deformation equations, despite the fact that Painlev\'e has never investigated isomonodromy. We prefer to use the term the Painlev\'e-type equations rather than the Schlesinger-type equations in order to express that our aim is to understand various higher-dimensional analogs of the Painlev\'e equations.}.
As for the 4-dimensional case, classification and derivation of the Painlev\'e-type equations from isomonodromic deformation have been completed~\cite{sakai_imd, kns, 2016arXiv160803927K, 2016arXiv160905263K, MR3740334}. 
According to their result, there are 40 types of 4-dimensional Painlev\'e-type equations.
Among these 40 types of equations, some of the equations coincide with 
 equations already known from different contexts.
Such well-known equations  along with new equations  are all organized in a unified way: isomonodromic deformation and degeneration.

There are problems to their classification.
They say that there are ``at most'' 40 types of 4-dimensional Painlev\'e-type equations.
There is no guarantee that these 40 equations are actually distinct.
Some Hamiltonians of these equations look very similar to each other.
Since the appearance of Hamiltonians or equations may change significantly by changes of variables, we can not classify equations by their appearances.
Intrinsic or geometrical studies of these equations may be necessary  to distinguish the Painlev\'e-type equations.
The aim of this paper is to initiate geometrical studies of 4-dimensional Painlev\'e-type equations starting from these concrete equations.

Let us first review how the 2-dimensional Painlev\'e equations are geometrically classified.
Okamoto  initiated the studies of the space of initial conditions of the Painlev\'e equations~\cite{MR614694}.
He constructed  the rational surfaces, whose points correspond to the germs of the meromorphic solutions of the Painlev\'e equations by resolving the singularities of the differential equations.
Sakai extracted the key features of the spaces of initial conditions and classified what he calls ``generalized Halphen surfaces'' with such features~\cite{MR1882403}.
The classification of the 2-dimensional difference Painlev\'e equations correspond to the classification of  generalized Halphen surfaces, and 8 types of such surfaces give the Painlev\'e differential equations.
Such surfaces are distinguished by their anticanonical divisors.
For the autonomous 2-dimensional Painlev\'e equations, the spaces of initial conditions are elliptic surfaces and their anticanonical divisors are one of the Kodaira types~\cite{MR3077699}.
We can say that the 2-dimensional autonomous Painlev\'e equations are characterized or distinguished by the corresponding Kodaira types.

It is expected to carry out a similar study for the 4-dimensional Painlev\'e-type equations.
However, a straightforward generalization of the 2-dimensional cases 
 seems to contain many difficulties.
One of the reason is that the classification of 4-folds is much  more difficult than that of surfaces.
We want to avoid facing such difficulties by considering the spectral curves of the autonomous limit of these Painlev\'e-type equations.
While the geometry
is made simple considerably in the autonomous limit,
 important characteristics of the original non-autonomous equations are retained.

Integrable systems are Hamiltonian systems on symplectic manifolds $(M^{2n},\omega,H)$ with $n$ functions  $f_1=H,\dots,f_n$ in involution $\{f_i,f_j\}=0$.
The regular level sets of the momentum map $F=(f_1,\dots,f_n)\colon M\to \C^n$ are the Liouville tori.
The image under $F$ of critical points are called the bifurcation diagram, and it is studied for characterizing integrable systems~\cite{MR2036760}.
However, studying the bifurcation diagrams may become complicated for higher dimensional cases.
When an integrable system is expressed in a Lax form, it is not difficult to determine the discriminant locus of spectral curves.
We can often find correspondence between the bifurcation diagrams and  the discriminant locus of spectral curves~\cite{MR1409362}. 
Therefore, we mainly study the degeneration of spectral curves in this paper.

Let $\{H_{i}\}_{i=1,\dots, g}$  be a set of functionally independent invariants of an integrable system $(M,\omega,H)$. 
In the 2-dimensional case, the following holds.
\begin{introthm}[cf.~Theorem~\ref{thm:genus1spectral}]
Each autonomous 2-dimensional Painlev\'e equation defines an elliptic surface, whose general fibers are spectral curves of the system.
The Kodaira types of the singular fibers at $H_{\rm J}=\infty$ are listed as follows.
 \begin{table}[H]
 \centering
  \begin{tabular}{|c|c|c|c|c|c|c|c|c|}\hline
  Hamiltonian & $H_{\rm VI}$ & $H_{\rm V}$ &$H_{\rm III(D_6)}$&$H_{\rm III(D_7)}$ &$H_{\rm III(D_8)}$&$H_{\rm IV}$&$H_{\rm II}$&$H_{\rm I}$\\ \hline
 Kodaira type & ${\rm I}_0^{*}$ & ${\rm I}_1^{*}$ & ${\rm I}_2^{*}$ & ${\rm I}_3^{*}$& ${\rm I}_4^{*}$& ${\rm IV}^{*}$& ${\rm III}^{*}$& ${\rm II}^{*}$\\ \hline
 Dynkin type & $D_4^{(1)}$ & $D_5^{(1)}$ & $D_6^{(1)}$ & $D_7^{(1)}$& $D_8^{(1)}$& $E_6^{(1)}$& $E_7^{(1)}$& $E_8^{(1)}$\\ \hline
  \end{tabular}
 \caption{The singular fiber at $H_{\rm J}=\infty$ of spectral curve fibrations of the autonomous 2-dimensional Painlev\'e equations.}
 \end{table}
\end{introthm}
It is well-known that the Dynkin's types in the above diagram  appear in the configurations of vertical leaves of the Okamoto's spaces of initial conditions~\cite{MR614694}.

For the autonomous 4-dimensional Painlev\'e-type equations, general invariant sets $\bigcap_{i=1,2}H_i^{-1}(h_i)$ are the 2-dimensional Liouville tori. 
Such Liouville tori are  the Jacobian varieties of the corresponding spectral curves.
Instead of studying the degenerations of 2-dimensional Liouville tori, we study the degenerations of the spectral curves of genus 2.
As an analogy of the above theorem for the 2-dimensional Painlev\'e equations, we  find the following:
\begin{main}[cf.~Theorem~\ref{thm:table}]
The spectral curves of the autonomous 4-dimensional Painlev\'e-type equations have the following types of generic degenerations as in Table~\ref{table:generic} in Section~\ref{sec:spectral}.
\end{main}
The table shows, for example, that generic degeneration of the spectral curves  of the autonomous matrix Painlev\'e equations $H_{\rm VI}^{\rm Mat}$, $H_{\rm V}^{\rm Mat}$, $H_{\rm III(D_6)}^{\rm Mat}$, $H_{\rm III(D_7)}^{\rm Mat}$ and $H_{\rm III(D_8)}^{\rm Mat}$
are 
 ${\rm I_0-I_0^*-1}$, ${\rm I_0-I_1^*-1}$, ${\rm I_0-I_2^*-1}$, ${\rm I_0-I_3^*-1}$ and ${\rm I_0-I_4^*-1}$ in Namikawa-Ueno's notation. 
Those of $H_{\rm IV}^{\rm Mat}$, $H_{\rm II}^{\rm Mat}$ and $H_{\rm I}^{\rm Mat}$ are
${\rm I_0-IV^*-1}$, ${\rm I_0-III^*-1}$ and ${\rm I_0-II^*-1}$.
Therefore, generic degeneration of the spectral curves of the autonomous 4-dimensional matrix Painlev\'e equations have one additional elliptic curve
to those counterpart of the 2-dimensional systems.

There are various ways to have integrable systems in general.
To identify  equations from different origins is often not easy; equations may change their appearance by transformations.
It is hoped that such intrinsic geometrical studies will be helpful for such identification problems. 
One possible approach using the degeneration of the Painlev\'e divisors\footnote{The Painlev\'e divisors introduced by Adler-Moerbeke\cite{MR999312,MR2095251} compactify affine Liouville tori.} is proposed in \cite{nakjmm}.
This direction will be investigated further in a forthcoming paper.

In recent years, Rains and his collaborators~\cite{MR3397405,2016arXiv160708876R,
2013arXiv1307.4033R} brought about crucial developments in the Painlev\'e-type  (difference) equations using noncommutative geometry.
In their theory, the anticanonical divisors of rational surfaces, determining the Poisson structures, are one of the key ingredients.
The relation between the anticanonical divisors in their work and the generic degeneration of spectral curves will be explained in 
future work.

\subsubsection*{Contents}  
The organization of this paper is as follows.
In  Section~\ref{sec:4-dim}, after summarizing preliminaries, we review the classification of the 4-dimensional Painlev\'e-type equations.
In  Section~\ref{sec:auto}, we consider the autonomous limit of these 40 equations.
In Section~\ref{sec:spectral}, we  study the generic degeneration of the spectral curves to characterize these integrable systems.
In Appendix \ref{ap:con}, we list integrals of the autonomous 4-dimensional Painlev\'e-type equations.
In Appendix \ref{graph}, we list the dual graph of the singular fibers appeared in our table.

\subsubsection*{Acknowledgments} 

\noindent 
I would like to express my utmost gratitude to 
 Prof.~Hiroshi Kawakami, Prof. ~Eric Rains, and Prof.~Hidetaka Sakai.
I am also grateful to
Prof.~Hayato Chiba,
Prof.~Yusuke Nakamura,
Prof.~Seiji Nishioka,
Prof.~Akishi Kato, 
Prof.~Hiroyuki Kawamuko,
Prof.~Kinji Kimura,
Prof.~Kazuki Hiroe,
Prof.~Shinobu Hosono,
Prof.~Takayuki Okuda,
and
Prof.~Ralph Willox
for valuable discussions and advices.

\section{Classification of 4-dimensional Painlev\'e-type equations}\label{sec:4-dim}
In this section, we review  some of the recent progresses in  classification of the 4-dimensional Painlev\'e-type equations,
 and introduce notation we use in this paper.
The contents of this section is a summary of the other papers~\cite{sakai_imd,kns,2016arXiv160803927K, 2016arXiv160905263K, MR3740334} and  references therein.

The Painlev\'e equations were found by Painlev\'e through his classification of the second order algebraic differential equations with the 
``Painlev\'e property''.
However, a straightforward application of Painlev\'e's classification method to higher-dimensional cases seems to face difficulties.\footnote{Works of Chazy~\cite{MR1555070} and Cosgrove~\cite{MR1738749,MR2220477} are famous in this direction.} 
Therefore other properties which characterize the Painlev\'e equations become important for further generalization.
The Painlev\'e equations 
can be expressed as Hamilton systems~\cite{Malmqist,MR614694}.
\begin{align*}
\frac{dq}{dt}=\frac{\partial H_{\rm J}}{\partial p},
\hspace{4mm}
\frac{dp}{dt}=-\frac{\partial H_{\rm J}}{\partial q}.
\end{align*}
We list the Hamiltonian functions for all the 2-dimensional Painlev\'e equations for later use. 
\begin{align*} 
 t(t-1)H_{\rm VI}\left( \begin{matrix} \alpha , \beta \\ \gamma , \epsilon \end{matrix};t;q,p\right)= {}& q(q-1)(q-t)p^2\\ 
 & + \{ \epsilon q(q-1)-(2\alpha +\beta +\gamma +\epsilon )q(q-t)+\gamma
 (q-1)(q-t)\} p\\ 
 & + \alpha (\alpha +\beta )(q-t),\\ 
 tH_{\rm V}\left( \begin{matrix} \alpha , \beta \\ \gamma \end{matrix} ;t;q,p\right)= {}&
  p(p+t)q(q-1) 
+\beta pq+\gamma p-(\alpha +\gamma )tq,\\ 
 H_{\rm IV}\left(\alpha , \beta;t;q,p\right)= {}&
pq(p-q-t)+\beta p+\alpha q,\\
 tH_{\rm III}(D_6)\left(\alpha , \beta ;t;q,p\right)= {}&
p^2q^2-(q^2-\beta q-t)p-\alpha q,\\ 
 tH_{\rm III}(D_7)\left(\alpha;t;q,p\right)= {}&  
p^2q^2+\alpha qp+tp+q,\quad
 tH_{\rm III}(D_8)\left(t;q,p\right)= 
p^2q^2+qp-q-\frac{t}{q},\\
 H_{\rm II}\left(\alpha;t;q,p\right)= {}& 
p^2-(q^2+t)p-\alpha q,\quad
 H_{\rm I}\left(t;q,p\right)= 
p^2-q^3-tq. 
\end{align*}

The Painlev\'e equations have another important aspect initiated by R.~Fuchs~\cite{RFuchs}.
Namely, they can be derived from (generalized) isomonodromic deformation of linear equations~\cite{MR630674}.
Furthermore, these eight types of the Painlev\'e equations are linked by  processes called degenerations.
In fact, $H_{\rm I},\cdots, H_{\rm V}$ can be  derived from $H_{\rm VI}$ through degenerations.

\vspace{3mm}
{\fontsize{11pt}{-3pt}
\begin{xy}
			{(0,0) 
				*{\begin{tabular}{c}
						$H_{\mathrm{VI}}$
					\end{tabular}
				}
			},	
		{\ar (5,0);(15,0)},		
			{(21,0) 
				*{\begin{tabular}{c}
					$H_{\mathrm{V}}$
								\end{tabular}
					}
			},	
		{\ar (25,0);(35,5)},	
		{\ar (25,0);(35,-5)},	
			{(44,5) 
				*{\begin{tabular}{c}
				$H_{\mathrm{III}}(D_6)$
											\end{tabular}
								}
						},	
			{(43,-5) 
				*{\begin{tabular}{c}
				$H_{\mathrm{IV}}$
				\end{tabular}
											}
									},	
		{\ar (53,5);(61,5)},
		{\ar (53,5);(61,-5)},	
		{\ar (53,-5);(61,-5)},						
		{(69,5) 
			*{\begin{tabular}{c}
			$H_{\mathrm{III}}(D_7)$
			\end{tabular}
										}
								},	
		{(69,-5) 
			*{\begin{tabular}{c}
			$H_{\mathrm{II}}$
			\end{tabular}
													}
											},	
		{\ar (76,5);(85,5)},
		{\ar (76,5);(85,-5)},	
		{\ar (76,-5);(85,-5)},									
	{(92,5) 
				*{\begin{tabular}{c}
				$H_{\mathrm{III}}(D_8)$
				\end{tabular}
											}
									},	
	{(92,-5) 
				*{\begin{tabular}{c}
				$H_{\mathrm{I}}$
				\end{tabular}
														}
												},																					
\end{xy}	
}

\subsubsection{The classification of Fuchsian equations and isomonodromic deformation}
If we fix the number of accessory parameters and identify the linear equations that transform into one another by Katz's operations (addition and middle convolutions)~\cite{MR1366651,MR2363121}, 
we have only finite types of linear equations.
We explain the notion of spectral type used in this paper.
We follow the notion used in Oshima~\cite{BB10865098} for Fuchsian linear equations, Kawakami-Nakamura-Sakai~\cite{kns} for unramified linear equations and Kawakami~\cite{2016arXiv160803927K} for ramified equations.
For the classification of linear equations, we need to discern the types of linear equations.
We review the local normal forms of linear equations, and introduce symbols to express such data.
We study linear systems of first-order equations
\begin{equation}\label{lde}
\frac{dY}{dx}=A(x)Y.
\end{equation}
We first consider the Fuchsian case where
\[
 A(x)=\sum_{i=1}^n\frac{A_i}{x-t_i}.
\]
We assume that each matrix $A_i$ is diagonalizable.
The equation can be transformed into
\begin{equation*}
\frac{d\hat{Y}(x)}{dx}=\frac{T^{(i)}}{x-t_{i}}\hat{Y}(x), \quad T^{(i)}\colon {\rm diangonal}
\end{equation*}
by a local transformation $Y=P(x)\hat{Y}$.
We express the multiplicity of the eigenvalues by a non-increasing sequence of numbers.
\begin{eg}
When $T^{(i)}=\operatorname{diag}(a,a,a,b,b,c)$, we write the multiplicity as $321$.
\end{eg}
Collecting such multiplicity data for all the singular points,
the spectral type of the linear equation is defined as the $n+1$-tuples of partitions of $m$, 
\[
\underbrace{
 m^1_1m^1_2\dots m^1_{l_1}
 }
 ,
 \underbrace{
  m^2_1\dots m^2_{l_2}
  }
  ,
   \dots ,
 \underbrace{
 m^n_1\dots m^n_{l_n}
 }
 ,
 \underbrace{
  m^{\infty}_1\dots m^{\infty}_{l_\infty}
  }
 ,
 \qquad \left(\sum_{j=1}^{l_i}m_j^i=m~~\text{for}~~1\leq\forall i\leq n
 ~\text{or}~i=\infty\right),
\]
where $m$ is the size of matrices. 
\begin{thm}[Kostov\cite{MR1877365}]
Irreducible Fuchsian equations with two accessory parameters result in one of the four types by successive additions and middle convolutions:
\begin{itemize}
\item $11,11,11,11$
\item $111,111,111\hspace{5mm} 22,1111,1111\hspace{5mm}33,222,111111.$
\end{itemize}
\end{thm}
\begin{rem}
Note that only the equation of the type $11,11,11,11$ has four singular points and the other three types have three singular points.
The three of the singular points can be fixed at $0,1,\infty$ by a M\"obius transformation.
Thus the three equations with only three singularities do not admit the continuous deformation of position of singularities.
Among these 4 types, only linear equation of type $11,11,11,11$ admit isomonodromic deformation, and it gives the sixth Painlev\'e equation $H_{\mathrm{VI}}$.
\end{rem}
The Katz's operations are important for studying Painlev\'e-type equations, because the following theorem holds. 
\begin{thm}[Haraoka-Filipuk~\cite{MR2363425}]\label{HF}
Isomonodromic deformation equations are invariant under Katz's operations.
\end{thm}
\begin{rem}
Katz's operations that do not change the type of the linear equation induce the corresponding B\"acklund transformations on the isomonodromic deformation equation.
In fact, all of the $D_{4}^{(1)}$-type affine Weyl group symmetry that $P_{\rm VI}$ possesses can be derived from Katz's operations and the Schlesinger transformations on the linear equations~\cite{nakmaster}.
\end{rem}
 
\subsubsection{The classification of Fuchsian linear equations and isomonodromic deformation}
The starting point of the classification of the 4-dimensional Painlev\'e-type equation is the following result.
\begin{thm}[Oshima~\cite{BB10865098}]
Irreducible Fuchsian equations with four accessory parameters result in one of the following 13 types by successive additions and middle convolutions\footnote{These operations are called the Katz' operations~\cite{MR1366651,MR2363121}}:
\begin{itemize}
\item$11,11,11,11,11$
\item $21,21,111,111\hspace{3mm} 31,22,22,1111\hspace{3mm}22,22,22,211$
\item $211,1111,1111 \hspace{2mm}221,221,11111\hspace{2mm}32,11111,11111$\\
$222,222,2211\hspace{2mm}33,2211,111111\hspace{2mm}44,2222,22211$\\
$44,332,11111111\hspace{2mm}55,3331,22222\hspace{2mm}{66,444,2222211}$.
\end{itemize}
\end{thm}
\begin{rem}
Note that the equation of type $11,11,11,11,11$ has five singular points, and the next three types have four singular points, and the rest nine types have three singular points.
The equation of type $11,11,11,11,11$ has two singularities
 to deform after fixing three of the singularities to $0,1,\infty$.
The next three types of equations with four singularities have one singularity 
 to deform after fixing three of the singularities to $0,1,\infty$.
The nine equations with only three singularities do not admit continuous isomonodromic deformation.
\end{rem}
The invariance of the isomonodromic deformation equation by Katz' operations is guaranteed by Haraoka-Filipuk~\cite{MR2363425}.
Sakai derived explicit Hamiltonians of the above four equations with four accessory parameters.
\begin{thm}[Sakai~\cite{sakai_imd}]
There are four 4-dimensional Painlev\'e-type equations governed by Fuchsian equations.
\begin{itemize}
\item The Garnier system in two variables (11,11,11,11,11).  %
\item The Fuji-Suzuki system (21,21,111, 111).
\item The Sasano system (31,22,22,1111). 
\item The Sixth Matrix Painlev\'e equation of size 2 (22,22,22,211).
\end{itemize}
\end{thm} 
In this paper, we call these 4 equations derived from Fuchsian equation the  ``source equations''.
\subsubsection{Degeneration scheme of 4-dimensional Painlev\'e-type equations}
Other 4-dimensional Painlev\'e-type equations we consider are derived from these source equations by degeneration process.
The degenerations corresponding to unramified linear equations are treated in Kawakami-Nakamura-Sakai~\cite{kns}.
The degenerations corresponding to ramified linear equations are treated in Kawakami~\cite{2016arXiv160803927K,2016arXiv160905263K,MR3740334}. 
There are 40 types of the 4-dimensional Painlev\'e-type equations with 16 partial differential equations corresponding to the differential Garnier equations and 24 ordinary differential equations.
Among 24 ordinary differential equations, 8 types are the matrix Painlev\'e equations.
According to~\cite{Rainsprivate},
the 7 types with the source equation 
$H_{\mathrm{Ss}}^{D_6}$ correspond to the symmetric $q$-difference Garnier equations,
 and the 9 types with the source equation 
 $H_{\mathrm{FS}}^{A_5}$ 
 correspond to the nonsymmetric $q$-difference
 Garnier equations in Rains~\cite{2013arXiv1307.4033R}.

\begin{table}[H]
\centering
\begin{tabular}{|c|c|c|c|c|} \hline 
& & Fuchsian & Non-Fuchsian  
&
$\#$
 \\ \hline\hline
&PDE & $H_{\mathrm{Gar}}^{1+1+1+1+1}$ & $H_{\mathrm{Gar}}^{2+1+1+1}$ \ $H_{\mathrm{Gar}}^{3+1+1}$ \ 
$H_{\mathrm{Gar}}^{2+2+1}$
$H_{\mathrm{Gar}}^{2+\frac{3}{2}+1}$\ $H_{\mathrm{Gar}}^{\frac{3}{2}+\frac{3}{2}+1}$\ $H_{\mathrm{Gar}}^{\frac{5}{2}+2}$ 
\rule[0mm]{0mm}{6mm}
&

\\
&
&
&
$H_{\mathrm{Gar}}^{4+1}$ \ $H_{\mathrm{Gar}}^{3+2}$ \ 
$H_{\mathrm{Gar}}^{5}$
\
$H_{\mathrm{Gar}}^{\frac{7}{2}+1}$
\
$H_{\mathrm{Gar}}^{\frac{3}{2}+3}$
\rule[0mm]{0mm}{5mm}
&
16
\\
Garnier
&& & $H_{\mathrm{Gar}}^{\frac{3}{2}+1+1+1}$ \ $H_{\mathrm{Gar}}^{\frac{5}{2}+1+1}$\
 $H_{\mathrm{Gar}}^{\frac{5}{2}+\frac{3}{2}}$\ $H_{\mathrm{Gar}}^{\frac{9}{2}}$
 \rule[0mm]{0mm}{5mm}
 &
\\ \cline{2-5}
&
ODE
& $H_{\mathrm{FS}}^{A_{5}}$ &  $H_{\mathrm{FS}}^{A_{4}}$ \ %
$H_{\mathrm{FS}}^{A_3}$
\ $H_{\mathrm{NY}}^{A_5}$
\ $H_{\mathrm{NY}}^{A_{4}}$\
 $H_{\mathrm{Suz}}^{2+\frac{3}{2}}$ \ $H_{\mathrm{KFS}}^{2+\frac{4}{3}}$ \ $H_{\mathrm{KFS}}^{\frac{3}{2}+\frac{4}{3}}$ \ $H_{\mathrm{KFS}}^{\frac{4}{3}+\frac{4}{3}}$
 \rule[0mm]{0mm}{6mm}
 &
 9
 \\
 \cline{3-5}
&  & $H_{\mathrm{Ss}}^{D_{6}}$ & $H_{\mathrm{Ss}}^{D_{5}}$ \ $H_{\mathrm{Ss}}^{D_{4}}$\ 
$H_{\mathrm{KSs}}^{2+\frac{3}{2}}$ \ $H_{\mathrm{KSs}}^{2+\frac{4}{3}}$ \ $H_{\mathrm{KSs}}^{2+\frac{5}{4}}$ \ $H_{\mathrm{KSs}}^{\frac{3}{2}+\frac{5}{4}}$
\rule[0mm]{0mm}{6mm}
&
7  \\
\hline
Matrix
&
ODE
&
 $H_{\mathrm{VI}}^{\rm Mat}$ 
 &
  $H_{\mathrm{V}}^{\rm Mat}$\
  $H_{\mathrm{IV}}^{\rm Mat}$ \
   $H_{\mathrm{II}}^{\rm Mat}$ \
    $H_{\mathrm{III}(D_{6})}^{\rm Mat}$ \ 
$H_{\mathrm{III}(D_{7})}^{\rm Mat}$\
 $H_{\mathrm{III}(D_{8})}^{\rm Mat}$ \
  $H_{\mathrm{I}}^{\rm Mat}$
  \rule[0mm]{0mm}{5mm}
  &
  8
  \\ \hline
\end{tabular}
\caption{The list of the 4-dimensional Painlev\'e-type equations}
\end{table}

As shown in the following diagram~\cite{MR3740334},
there are 4 series of degeneration diagram corresponding to 4 ``source equations''.
Explicit forms of the Hamiltonians and the Lax pairs can be found in ~\cite{kns,2016arXiv160803927K, 2016arXiv160905263K, MR3740334}.

{\fontsize{11pt}{-3pt}
\begin{xy}
			{(0,0) 
				*{\begin{tabular}{c}
						$H_{\mathrm{Gar}}^{1+1+1+1+1}$
					\end{tabular}
				}
			},
	{(0,-24) 
					*{\begin{tabular}{c}
							$H_{\mathrm{FS}}^{A_5}$
						\end{tabular}
					}
				},
	{(0,-40) 
						*{\begin{tabular}{c}
								$H_{\mathrm{Ss}}^{D_6}$
							\end{tabular}
						}
					},					
		{\ar (7,0);(14,0)},	
		{\ar (7,-24);(14,-24)},
		{\ar (7,-24);(14,0)},	
		{\ar (7,-24);(14,-32)},		
		{\ar (7,-40);(14,-40)},	
		{\ar (7,-40);(14,-32)},	
			{(22,0) 
					*{\begin{tabular}{c}
							$H_{\mathrm{Gar}}^{2+1+1+1}$
								\end{tabular}
							}
						},
			{(22,-24) 
					*{\begin{tabular}{c}
							$H_{\mathrm{FS}}^{A_4}$
					\end{tabular}
									}
								},		
			{(22,-32) 
				*{\begin{tabular}{c}
				$H_{\mathrm{NY}}^{A_5}$
				\end{tabular}
												}
											},
			{(22,-40) 
				*{\begin{tabular}{c}
				$H_{\mathrm{Ss}}^{D_5}$
				\end{tabular}
							}										},													
			{\ar (30,0);(36,7)},
			{\ar (30,0);(36,0)},
			{\ar (30,0);(36,-7)},
			{\ar (30,-24);(36,7)},
			{\ar (30,-24);(36,-32)},	
			{\ar (30,-24);(36,-24)},
			{\ar (30,-32);(36,-32)},
			{\ar (30,-32);(36,-7)},
			{\ar (30,-40);(36,-40)},		
			{\ar (30,-40);(36,-32)},				
			{(44,8) 
									*{\begin{tabular}{c}
									$H_{\mathrm{Gar}}^{3+1+1}$
										\end{tabular}
										}
										},				
			{(44,0) 
						*{\begin{tabular}{c}
						$H_{\mathrm{Gar}}^{2+2+1}$
							\end{tabular}
							}
							},		
			{(46,-8) 
						*{\begin{tabular}{c}
						$H_{\mathrm{Gar}}^{\frac{3}{2}+1+1+1}$
										\end{tabular}
										}
										},	
			{(46,-24) 
						*{\begin{tabular}{c}
						$H_{\mathrm{FS}}^{A_3}$
						\end{tabular}
									}
									},	
			{(46,-32) 
						*{\begin{tabular}{c}
						$H_{\mathrm{NY}}^{A_4}$
						\end{tabular}
									}
									},		
		{(46,-40) 
								*{\begin{tabular}{c}
								$H_{\mathrm{Ss}}^{D_4}$
								\end{tabular}
											}
											},																				
		{\ar (55,7);(63,11)},	
		{\ar (55,7);(63,4)},
		{\ar (55,7);(63,-4)},
		{\ar (55,0);(63,11)},	
		{\ar (55,0);(63,4)},	
		{\ar (55,0);(63,-11)},		
		{\ar (55,-7);(63,-4)},	
		{\ar (55,-7);(63,-11)},	
		{\ar (55,-24);(63,4)},
		{\ar (55,-24);(63,-4)},
		{\ar (55,-24);(63,-24)},
		{\ar (55,-32);(63,-4)},	
		{\ar (55,-40);(63,-40)},
		{\ar (55,-40);(63,-4)},								
		{(68,12) 
						*{\begin{tabular}{c}
						$H_{\mathrm{Gar}}^{4+1}$				\end{tabular}
								}
										},				
		{(68,4) 
								*{\begin{tabular}{c}
								$H_{\mathrm{Gar}}^{3+2}$
									\end{tabular}
									}
									},		
		{(70,-4) 
					*{\begin{tabular}{c}
					$H_{\mathrm{Gar}}^{\frac{5}{2}+1+1}$
								\end{tabular}
									}
									},	
		{(71,-12) 
							*{\begin{tabular}{c}
							$H_{\mathrm{Gar}}^{2+\frac{3}{2}+1}$
										\end{tabular}
											}
											},
	{(70,-24) 
				*{\begin{tabular}{c}
				$H_{\mathrm{Suz}}^{2+\frac{3}{2}}$
											\end{tabular}
												}
												},
	{(70,-40) 
			*{\begin{tabular}{c}
			$H_{\mathrm{KSs}}^{2+\frac{3}{2}}$
										\end{tabular}
													}
													},																					
		{\ar (77,12);(86,16)},	
		{\ar (77,12);(86,8)},	
		{\ar (77,4);(86,16)},	
	    {\ar (77,4);(86,0)},	
	    {\ar (77,4);(86,-8)},		
	    {\ar (77,-4);(86,8)},	
	    {\ar (77,-4);(86,-8)},
	    {\ar (77,-12);(86,0)},		
	    {\ar (77,-12);(86,-8)},		
	    {\ar (77,-12);(86,-16)},
	    {\ar (77,-24);(86,-24)},
	    {\ar (77,-24);(86,16)},		
	    {\ar (77,-40);(86,-40)},				
		{(91,16) 
						*{\begin{tabular}{c}
						$H_{\mathrm{Gar}}^{5}$
						\end{tabular}
										}
														},		
		{(91,8) 
					*{\begin{tabular}{c}
					$H_{\mathrm{Gar}}^{\frac{7}{2}+1}$
											\end{tabular}
											}
											},		
		{(91,0) 
					*{\begin{tabular}{c}
					$H_{\mathrm{Gar}}^{3+\frac{3}{2}}$
						\end{tabular}
									}
									},	
		{(91,-8) 
				*{\begin{tabular}{c}
				$H_{\mathrm{Gar}}^{\frac{5}{2}+2}$
					\end{tabular}
								}
							},
		{(93,-16) 
						*{\begin{tabular}{c}
						$H_{\mathrm{Gar}}^{\frac{3}{2}+\frac{3}{2}+1}$
							\end{tabular}
										}
									},
		{(91,-24) 
			*{\begin{tabular}{c}
				$H_{\mathrm{KFS}}^{2+\frac{4}{3}}$
							\end{tabular}
										}
									},	
	{(91,-40) 
		*{\begin{tabular}{c}
			$H_{\mathrm{KSs}}^{\frac{3}{2}+\frac{3}{2}}$
								\end{tabular}
											}
										},															
	{\ar (100,16);(108,8)},
	{\ar (100,8);(108,8)},
	{\ar (100,0);(108,8)},
	{\ar (100,0);(108,-8)},
	{\ar (100,-8);(108,8)},
	{\ar (100,-8);(108,-8)},									{\ar (100,-16);(108,-8)},	
	{\ar (100,-24);(108,-24)},	
	{\ar (100,-40);(108,-40)},	
	{(113,8) 
						*{\begin{tabular}{c}
						$H_{\mathrm{Gar}}^{\frac{9}{2}}$
												\end{tabular}
												}
												},	
	{(113,-8) 
					*{\begin{tabular}{c}
					$H_{\mathrm{Gar}}^{\frac{5}{2}+\frac{3}{2}}$
						\end{tabular}
				}
				},					
	{(113,-24) 
			*{\begin{tabular}{c}								$H_{\mathrm{KFS}}^{\frac{3}{2}+\frac{4}{3}}$
			\end{tabular}												}	
					},	
	{(113,-40) 
		*{\begin{tabular}{c}								$H_{\mathrm{KSs}}^{2+\frac{5}{4}}$
				\end{tabular}												}	
						},		
	{\ar (120,-24);(128,-24)},	
	{\ar (120,-40);(128,-40)},									
	{(133,-24) 
		*{\begin{tabular}{c}								$H_{\mathrm{KFS}}^{\frac{4}{3}+\frac{4}{3}}$
			\end{tabular}												}	
						},	
	{(133,-40) 
			*{\begin{tabular}{c}								$H_{\mathrm{KSs}}^{\frac{3}{2}+\frac{5}{4}}$
				\end{tabular}												}	
							},													
			{(0,-55) 
				*{\begin{tabular}{c}
						$H^{\mathrm{VI}}_{\mathrm{Mat}}$
					\end{tabular}
				}
			},	
		{\ar (5,-55);(15,-55)},		
			{(23,-55) 
				*{\begin{tabular}{c}
					$H^{\mathrm{V}}_{\mathrm{Mat}}$
								\end{tabular}
					}
			},	
		{\ar (28,-55);(38,-50)},	
		{\ar (28,-55);(38,-60)},	
			{(46,-50) 
				*{\begin{tabular}{c}
				$H^{\mathrm{III}(D_6)}_{\mathrm{Mat}}$
											\end{tabular}
								}
						},	
			{(46,-60) 
				*{\begin{tabular}{c}
				$H^{\mathrm{IV}}_{\mathrm{Mat}}$
				\end{tabular}
											}
									},	
		{\ar (53,-50);(61,-50)},
		{\ar (53,-50);(61,-60)},	
		{\ar (53,-60);(61,-60)},						
		{(69,-50) 
			*{\begin{tabular}{c}
			$H^{\mathrm{III}(D_7)}_{\mathrm{Mat}}$
			\end{tabular}
										}
								},	
		{(69,-60) 
			*{\begin{tabular}{c}
			$H^{\mathrm{II}}_{\mathrm{Mat}}$
			\end{tabular}
													}
											},	
		{\ar (76,-50);(85,-50)},
		{\ar (76,-50);(85,-60)},	
		{\ar (76,-60);(85,-60)},									
	{(92,-50) 
				*{\begin{tabular}{c}
				$H^{\mathrm{III}(D_8)}_{\mathrm{Mat}}$
				\end{tabular}
											}
									},	
	{(92,-60) 
				*{\begin{tabular}{c}
				$H^{\mathrm{I}}_{\mathrm{Mat}}$
				\end{tabular}
														}
												},																					
\end{xy}	
}

\begin{rem}
The names of these Hamiltonians are temporal.
As Sakai~\cite{MR1882403} labeled the 2-dimensional systems by the  types of the anticanonical divisors of the compactified spaces of initial conditions,
it may be natural to label the  4-dimensional systems from geometrical characterization. 
\end{rem}

\begin{rem}
While Kawakami-Nakamura-Sakai~\cite{kns} studied the degeneration from Fuchsian types, the theory of unramified non-Fuchsian linear equations has developed.
Hiroe and Oshima  classified all the unramified linear equations with 4 accessory parameters up to some transformations~\cite[Thm3.29]{MR3077686}.
Yamakawa proved that analogous theorem of Haraoka-Filpuk~\cite{MR2363425} holds for unramified non-Fuchsian equations~\cite{MR3642539}.
By comparing the results,
 all the unramified equations with 4 accessory parameters come from the 4 Fuchsian source equations by degenerations, and the list of 4-dimensional Painlev\'e-type equations corresponding to the unramified linear equations is complete.
 To the author's knowledge, it is still an open question whether linear equations of ramified type with 4 accessory parameters can be reduce to this list of 40 types.
 Rains' work~\cite{RainsTexas} might be conclusive in this direction.
\end{rem}

\begin{rem}
According to the classification of the Lax pairs by
Rains~\cite{RainsTexas}, there are 40 types of 4-dimensional families which admit continuous deformations, comprising of 16 differential Garnier equations, 7 symmetric $q$-difference Garnier equations, 9 nonsymmetric $q$-Garnier equations and 8 matrix Painlev\'e equations.
The numbers  match with Kawakami's list~\cite{MR3740334}.
\end{rem}

\begin{rem}
Some of these 40 equations look similar to each other.
For instance, H.~Chiba pointed out that $H_{\mathrm{Gar},t_1}^{4+1}$ and $\tilde{H}_{\rm II}^{\rm Mat}$ look almost the same after a symplectic transformations.\footnote{
The following transformation to the $H_{\rm II}^{\rm Mat}$ in \cite{kns}
 yields $\tilde{H}_{\rm II}^{\rm Mat}$:
$
p_1\to 2 p_1,q_1\to \frac{q_1}{2},p_2\to -q_2,q_2\to p_2,\kappa _1\to -\theta _0-\kappa _2.
$
}
\begin{align*}
	H_{\mathrm{Gar},t_1}^{4+1}
	=
	&
	p_1^2
	- \left(q_1^2+t_1\right)p_1
	+\kappa _1 q_1
	+p_1 p_2
	+p_2 q_2 \left(q_1-q_2+t_2\right)
	+\theta _0 q_2,
	\\
	\tilde{H}_{\rm II}^{\rm Mat}
	=
	&
	p_1^2
	- \left(\frac{q_1^2}{4}+t\right)p_1
	- \left(\theta_0+\frac{\kappa _2}{2}\right)q_1
	+p_1 p_2
	+p_2q_2 \left(q_1-q_2\right) 
	+\theta_0q_2.
\end{align*}
One of the key motivations of the present paper is to geometrically distinguish such cases.
We will show in chapter~\ref{sec:spectral} that the types of generic degeneration of spectral curves are different.\footnote{It is hoped to prove a theorem stating ``If the types of generic degeneration of the spectral curves are different, the corresponding space of initial conditions are not isomorphic, or, do not admit biregular map from one another.''}
\end{rem}
\begin{rem}
Some of the equations in the list, such as the Noumi-Yamada systems,  had been known from different context. 
See \cite{kns} for the references for other derivations.
\end{rem}

\section{Autonomous limit of Painlev\'e type equations}\label{sec:auto}
In the previous section, we saw that there are 40 types of 4-dimensional Painlev\'e-type equations.
In this section,
we consider the autonomous limit of these 40  equations
 by taking the isospectral limit of the isomonodromic deformation equations.
Using the Lax pair, we obtain two functionally independent invariants for each system.
Therefore, the autonomous limit of 4-dimensional Painlev\'e-type equations are integrable in Liouville's sense.

\subsection{Integrable system and Lax pair representation}
Let us  recall the definition of integrability in Liouville's sense. 
A Hamiltonian system is a triple $(M,\omega,H)$, where $(M,\omega)$ is a symplectic manifold and $H$ is a Hamiltonian function on $M$; $\iota_{X_H}\omega=dH$.
A function $f$ Poisson commutes with the Hamiltonian $H$, that is $\{f,H\}=0$, if and only if $f$ is constant along integrable curve of the Hamiltonian vector field $X_{H}$.
Such a function $f$ is called a conserved quantity or first integral.
A Hamiltonian system is (completely) integrable in Liouville's sense if it possesses $n\coloneqq\frac{1}{2}\operatorname{dim}M$ independent integrals of motion, $f_1=H,$ $f_2,\dots,f_n$, which are pairwise in involution with respect to the Poisson bracket; $\{f_i,f_j\}=0$ for all $i,j$.
This definition of integrability is motivated by Liouville's theorem.
Let  $(M,\omega,H)$ be a real integrable system of dimension $2n$ with integral of motion $f_1,\dots,f_n$, and let $c\in\R^n$ be a regular value of $f=(f_1,\dots,f_n)$.
Liouville's theorem states that any compact component of the level set $f^{-1}(c)$ is a torus. 
The complex Liouville theorem is also known~\cite{MR2095251}.

Many integrable systems are known to have Lax pair expressions:
\begin{align}\label{eq:Lax}
\frac{dA(x)}{dt}+[A(x),B(x)]=0,
\end{align}
where $A(x)$ and $B(x)$ are $m$ by $m$ matrices and $x$ is a spectral parameter.
From this differential equation, $\operatorname{tr}\left(A(x)^k\right)$ are conserved quantities of the system:
\begin{align*}
\frac{d}{dt}\operatorname{tr}\left(A(x)^k\right)
=&\operatorname{tr}\left(k\left[B(x),A(x)\right]A(x)^{k-1}\right)
= 0.
\end{align*}
Therefore, the eigenvalues of $A(x)$ are all conserved quantities since the coefficients of the characteristic polynomial are expressible in terms of these traces.
In fact, the Lax pair is equivalent to the following isospectral problem:
\begin{align*}
\begin{dcases}
A(x)= Y A_0(x)Y^{-1},\\
\frac{d Y}{d t}=B(x)Y,
\end{dcases}
\end{align*}
where $A_0(x)$ is a matrix satisfying $\frac{d A_0(x)}{dt}=0$.
The curve defined by the characteristic polynomial is called  the spectral curve:
\begin{align*}
\operatorname{det}\left(yI_{m}-A(x)\right)=0.
\end{align*}
\subsection{Isomonodromic deformation to isospectral deformation}
The isomonodromic problems have the following forms:
\begin{align*}
\begin{dcases}
\frac{\partial Y}{\partial x}=A(x,t)Y, \\ 
\frac{\partial Y}{\partial t}=B(x,t)Y,
\end{dcases}
\end{align*}
and the deformation equation is expressed as
\begin{align}\label{eq:deltaLax}
\frac{\partial A(x,t)}{\partial t}-\frac{\partial B(x,t)}{\partial x}+[A(x,t),B(x,t)]=0.
\end{align}
We find the similarities in isospectral and isomonodromic problems; the only difference is the existence of the term $\frac{\partial B}{\partial x}$ in isomonodromic deformation equation.
In fact, we can consider isospectral problems as the special limit of isomonodromic problem with a parameter $\delta$.
We restate the isomonodromic problem as follows\footnote{Here $\delta$ plays the role of $\lambda$ in Deligne's lambda connections~\cite{MR1492538}. In some literature such as Levin-Olshanetsky-Zotov~\cite{MR2249796}, it is customary to use $\kappa$ instead of $\delta$.}:
\begin{align*}
\begin{dcases}
\delta\frac{\partial Y}{dx} = A(x,\tilde{t})Y, \\ 
\frac{\partial Y}{\partial t} = B(x,\tilde{t})Y,
\end{dcases}
\end{align*}
where $\tilde{t}$ is a variable which satisfies $\frac{d\tilde{t}}{dt}=\delta$.
The integrability condition $\frac{\partial^2 Y}{\partial x\partial t}=\frac{\partial^2 Y}{\partial t\partial x}$ is equivalent to the following:
\begin{align}\label{eq:delta-deform}
\frac{\partial A(x,\tilde{t})}{\partial t}-\delta\frac{\partial B(x,\tilde{t})}{\partial x}+[A(x,\tilde{t}),B(x,\tilde{t})]=0.
\end{align}
The case when $\delta=1$ is the usual one.\footnote{Adams-Harnad-Hurtubise~\cite{MR953833} and Adams-Harnad-Previato~\cite{MR1077962} studied finite dimensional integrable systems by embedding them into rational coadjoint orbits of loop algebras.
Harnad~\cite{MR1309553} further generalized their theory as applicable to the isomonodromic systems.
Such nonautonomous isomonodromic systems are obtained by
  identifying the time flows of the integrable system with parameters determining the moment map.}
When $\delta=0$, the term $\delta\frac{\partial B}{\partial x}$ drops off from the deformation equation and we have a Lax pair in a narrow sense.\footnote{In other words, we mean a Lax pair in the sense of integrable systems. The isomonodromic problems are often called as Lax pairs, but they do not give first integrals.}
The deformation equation~\ref{eq:deltaLax} with $\delta$ is solved by a Hamiltonian $H(\delta)$.
When $\delta=1$, the Hamiltonian $H(1)$ is equal to the original Hamiltonian of the isomonodromic problem.
Therefore, $H(\delta)$ is a slight modification of the Hamiltonian.
When $\delta=0$, $H(0)$ is a conserved quantity of the system.
\begin{rem}
Isomonodromic equations are flows on moduli space of connections~\cite{MR2289083}.
Isospectral limit correspond to $\lambda\to 0$ limit of moduli of $\lambda$-connections~\cite{MR1492538} to moduli of Higgs bundles.
The Painlev\'e-type equations become the Hitchin systems~\cite{MR885778} at the limit. 
\end{rem}

Taking the isospectral limit of 8 types of 2-dimensional Painlev\'e equations, we can state the following classically-known result.
\begin{prop}
As the autonomous limits of 2-dimensional Painlev\'e equations, we obtain 8 types of integrable systems
with a first integral for each system\footnote{These first integrals are the autonomous Hamiltonians. They are rational in the phase variables $q,p$.}.
\end{prop}
\begin{proof}
We take the second Painlev\'e equation as an example to demonstrate a proof. 
Proofs of the other equations are similar.
\begin{align*}
\begin{dcases}
\delta\frac{\partial Y}{\partial x}=A(x,\tilde{t})Y,\hspace{5mm}
A(x,\tilde{t})=\left(A_{\infty}^{(-3)}(\tilde{t})x^2+A_{\infty}^{(-2)}(\tilde{t})x+A_{\infty}^{(-1)}(\tilde{t})\right),\\
\frac{\partial Y}{\partial t}=B(x,\tilde{t})Y \hspace{5mm}
B(x,\tilde{t})=\left(A_{\infty}^{(-3)}(\tilde{t})x+B_1(\tilde{t})\right),
\end{dcases}
\end{align*}
where
\begin{align*}
& \hat{A}_{\infty}^{(-3)}(\tilde{t})=
\begin{pmatrix}
0 & 0\\
0 & 1
\end{pmatrix},\hspace{3mm}
\hat{A}_{\infty}^{(-2)}(\tilde{t})=
\begin{pmatrix}
0 & 1\\
-p+q^2+\tilde{t} & 0
\end{pmatrix},\\
& \hat{A}_{\infty}^{(-1)}(\tilde{t})=
\begin{pmatrix}
-p+q^2+\tilde{t} & q\\
(p-q^2-\tilde{t})q-\kappa_2 & p-q^2
\end{pmatrix},\hspace{2mm}
\hat{B}_{1}(\tilde{t})=
\begin{pmatrix}
q & 1\\
p-q^2-\tilde{t} & 0
\end{pmatrix},\\
& A_{\infty}^{(-i)}=U^{-1}\hat{A}_{\infty}^{(-i)}U\hspace{3mm}
{\rm for}\ i=1,2,3,\hspace{3mm}
B_1=U^{-1}\hat{B}_{1}U,\hspace{3mm}
U=\begin{pmatrix}
u & 0\\
0 & 1
\end{pmatrix}.
\end{align*}
The deformation equation (\ref{eq:delta-deform}) is equivalent to the following differential equations.
\begin{align*}
\frac{dq}{dt}= 2p-q^2-\tilde{t},\hspace{5mm}
\frac{dp}{dt}= 2pq+\delta-\kappa_1,\hspace{5mm}
\frac{du}{dt}= 0.
\end{align*}
The first two equations are equivalent to the Hamiltonian system
\[
\frac{dq}{dt}=\frac{\partial H_{\rm II}(\delta)}{\partial p},\hspace{3mm}
\frac{dp}{dt}=-\frac{\partial H_{\rm II}(\delta)}{\partial q},
\]
with the Hamiltonian 
$
H_{\rm II}(\delta):=p^2-(q^2+\tilde{t})p+(\kappa_1-\delta)q.
$
When $\delta=1$, it is the usual Hamiltonian of $H_{\rm II}$.
Moreover, when $\delta\neq 0$, we can normalize to the $\delta=1$ case.
Taking the limit $\delta\to 0$, we obtain an autonomous system with a Hamiltonian 
$
H_{\rm II}(0)=p^2-(q^2+\tilde{t})p+\kappa_1 q.
$
Since it is an autonomous system, the Hamiltonian is a first integral.
The dimension of the phase space is two,  the number of first integrals is half of the dimension.
Therefore, the autonomous second Painlev\'e equation is integrable in  Liouville's sense.
The Lax pair\footnote{We rewrite $A(x)=A(x,\tilde{t})$ and $B(x)=B(x,\tilde{t})$.}
and the spectral curve of the autonomous second Painlev\'e equation are
\begin{align*}
&\frac{d A(x)}{dt}+[A(x),B(x)]=0,
\\
&\det(yI-A(x))=y^2-(x^2+\tilde{t})y-\kappa_1 x-H_{\rm II}(0)=0.
\label{eq:spP2}
\end{align*}
\end{proof}
\begin{rem}
For the 2-dimensional cases, parameters of the Painlev\'e equations can be thought as roots of affine root systems, and $\delta$ corresponds to the null root~\cite{MR1882403,MR3077699}.
\end{rem}
\subsection{The autonomous limit of the 4-dimensional Painlev\'e-type equations}
We can also consider such autonomous limit for higher dimensional Painlev\'e-type equations.
From the coefficients of the spectral curves, we obtain first integrals.
\begin{thm}
As the autonomous limits of 4-dimensional Painlev\'e-type equations, we obtain 40 types of integrable systems
with two functionally independent first integrals for each system\footnote{These invariants are rational in the phase variables $q_1,p_1,q_2,p_2$.}.
\end{thm}
\begin{proof}
One of the simplest 4-dimensional Painlev\'e-type equation is the first matrix Painlev\'e equation~\cite{2016arXiv160803927K}.
The linear equation is given by 
\begin{align*}
\frac{d A(x)}{dt}+[A(x),B(x)]=0,
\hspace{3mm}
 A(x)=\left(A_{0}x^2+A_{1}x+A_{2}\right),
\hspace{3mm}
B(x)=
A_{0}x+B_1.
\end{align*}
where 
\begin{align*}
& A_{0}=
\begin{pmatrix}
O_2 & I_2\\
O_2 & O_2
\end{pmatrix},\hspace{3mm}
A_{1}=
\begin{pmatrix}
O_2 & Q\\
I_2 & O_2
\end{pmatrix},\hspace{3mm}
A_{2}=
\begin{pmatrix}
-P & Q^2+\tilde{t}I_2\\
-Q & P
\end{pmatrix},
\hspace{2mm}
B_1=
\begin{pmatrix}
O_2 & 2 Q\\
I_2 & O_2
\end{pmatrix},
\\
&
 Q=
\begin{pmatrix}
q_1 & u\\
-q_2/u & q_1
\end{pmatrix}, \hspace{2mm}
P=
\begin{pmatrix}
p_1/2 & -p_2u \\
(p_2q_2-\kappa_2)/u & p_1/2
\end{pmatrix},\hspace{2mm}
O_2=\begin{pmatrix}
0 & 0\\
0 & 0
\end{pmatrix},\hspace{2mm}
I_2=
\begin{pmatrix}
1 & 0\\
0 & 1
\end{pmatrix}.
\end{align*}
The spectral curve is defined by the characteristic polynomial of the matrix $A(x)$;
\begin{align*}
{\rm det}\left(yI_4-A(x)\right)=&{y}^{4}- \left( 2\,{x}^{3}+2\tilde{t}x+h \right) {y}^{2}+{x}^{6}+2\tilde{t}{x}^{4}
+h{x}^{3}+\tilde{t}^{2}{x}^{2}+ \left(\tilde{ t} h-\kappa_2^{2} \right) x+g.
\end{align*}
The explicit forms of $h$ and $g$ are
\begin{align*}
h\coloneqq 
H_{\rm I}^{\rm Mat}=
&\operatorname{tr} \left(P^2-Q^3-\tilde{t} Q \right)
=-2 p_2 \left(p_2 q_2-\kappa
   _2\right)+\frac{p_1^2}{2}-2 q_1 \tilde{t}-2 q_1
   \left(q_1^2-q_2\right)+4 q_1 q_2,\notag\\
g\coloneqq 
G_{\rm I}^{\rm Mat}=&q_2
   \left(p_1 p_2+3 q_1^2-q_2+\tilde{t}\right){}^2
   -\kappa_2 p_1 \left(p_1 p_2+3 q_1^2-q_2+\tilde{t}\right)
   -2 \kappa_2^2 q_1.
\end{align*}
Since 
$h$ and $g$
 are coefficient of the spectral curve, they are invariants of the autonomous system.
We can also check that 
$h$ and $g$
 are conserved  by direct computation:
\[
\dot{h}=X_{h}h=\{h,h\}=0,
\hspace{3mm}
\dot{g}=X_{h}g=\{g,h\}=0,
\]
where $X_{h}$ is the Hamiltonian vector field associated to the Hamiltonian $h$.
The Poisson bracket in the above equations is defined by
\[
\{F,G\}\coloneqq\sum_{i=1}^{2}
\left(
\frac{\partial F}{\partial q_i}\frac{\partial G}{\partial p_i}
-\frac{\partial G}{\partial q_i}\frac{\partial F}{\partial p_i}
\right).
\]
Since
\[
\operatorname{rank}
\begin{pmatrix}
\frac{\partial h}{\partial q_1} 
& \frac{\partial h}{\partial p_1}
& \frac{\partial h}{\partial q_2}
& \frac{\partial h}{\partial p_2}
\\
\frac{\partial g}{\partial q_1} 
& \frac{\partial g}{\partial p_1}
& \frac{\partial g}{\partial q_2}
& \frac{\partial g}{\partial p_2}
\end{pmatrix}
=2
\]
for the  general value of $(q_1,p_1,q_2,p_2)$, we have two functionally independent invariants of the system.
Thus the autonomous Hamiltonian system with the Hamiltonian $H_{\rm I}^{\rm Mat}$ is integrable in Liouville's sense. 

From similar direct computations, we obtain the desired results for all the rest of 4-dimensional Painlev\'e-type equations.
We  list functions in involution for the ramified types in Appendix~\ref{ap:con}.
These spectral curves and conserved quantities can be calculated from the data in the papers~\cite{sakai_imd,kns,2016arXiv160803927K,2016arXiv160905263K,MR3740334}. 
The only troublesome part is to find appropriate modifications of the Hamiltonians in the presence of $\delta$.
The other parts are straightforward.
\end{proof}
\begin{rem}
It is an interesting problem to study the invariant surfaces defined by $H^{-1}(c_1)\cap G^{-1}(c_2)\subset \C^4$ for $c_1, c_2\in \C$, where $H$ and $G$ are  functionally independent invariants of the system.
As in the case of other integrable systems~\cite{MR2095251}, these Liouville tori can be completed into Abelian surfaces by adjoining the Painlev\'e divisors~\cite{nakjmm}.
\end{rem}

\section{Degeneration of spectral curves\label{sec:spectral}}
This section is the main part of this paper.
We study generic degeneration of spectral curves of the autonomous 4-dimensional Painlev\'e-type equations, aiming to characterize these systems.
\subsection{Genus 1 fibration and Tate's algorithm}
Before discussing the genus 2 cases, corresponding to the autonomous limit of 4-dimensional Painlev\'e-type equations, we discuss the genus 1 cases, corresponding to the autonomous 2-dimensional Painlev\'e equations. 

Let us recall some of the basics we need.
We can construct the Kodaira-N\'eron model of an elliptic curve $E$ over 
$\A^1=\operatorname{Spec}(\C[h])$.
The possible types of singular fibers of elliptic surfaces were classified by Kodaira~\cite{MR0184257}.
Tate's algorithm provides a way to determine the Kodaira type of singular fibers without actually resolving the singularities~\cite{MR0393039}.

We consider an elliptic curve $E$ over 
$\A^1$
 with a section in the Weierstrass form:
\begin{equation}\label{eq:wei}
y^2=x^3+a(h)x+b(h),\hspace{5mm} a(h),b(h)\in \C[h].
\end{equation}
We may assume that for $a(h)$ and $b(h)$, the polynomials  $l(h)$ such that $l(h)^4|a(h),\ l(h)^6|b(h)$ are only constants.
Otherwise, we may divide both sides of the equations by $l(h)^6$ and replace $x, y$ by $x/l(h)^2, y/l(h)^3$ if necessary. 
Let $X_1$ be the affine surface defined by the equation (\ref{eq:wei}):
\[
X_1=\left\{(x,y,h)\in \A^2\times \A^1\relmiddle| y^2=x^3+a(h)x+b(h)\right\}.
\]
A general fiber of the projection $\varphi_1\colon X_1\to \A^1$ is an affine part of an elliptic curve.
Let $n$ be the minimal positive integer satisfying $\operatorname{deg}a(h)\leq4n$ and $\operatorname{deg}b(h)\leq6n$.
Dividing  equation (\ref{eq:wei}) by $h^{6n}$ and replacing
$
\tilde{x}=x/h^{2n},\
\tilde{y}=y/h^{3n},\
\tilde{h}=1/h,
$
we obtain the ``$\infty$-model'':
\begin{equation}\label{eq:wei2}
\tilde{y}^2=\tilde{x}^3+\bar{a}(\tilde{h})\tilde{x}+\bar{b}(\tilde{h}),
\hspace{3mm}
\bar{a}(\tilde{h}), \bar{b}(\tilde{h})\in\C[\tilde{h}]
\end{equation}
where $\bar{a}(\tilde{h})=a(h)/h^{4n},\ \bar{b}(\tilde{h})=b(h)/h^{6n}$ are polynomials in $h$.
Let $X_2$ be the affine surface defined by equation (\ref{eq:wei2}).
Let $\overline{X_1}$ and $\overline{X_2}$ be the projectivized surfaces in $\pr^2\times \mathbb{A}^1$;
$
\overline{X_i}\subset \mathbb{P}^2\times \mathbb{A}^1,\
\overline{\varphi_i}\colon \overline{X_i}\to\A^1$.
We glue $\overline{X_1}$ and $\overline{X_2}$ by identifying $(x,y,h)$ and $(\tilde{x},\tilde{y},\tilde{h})$ by the equations
above.
Let us denote the surface obtained this way by $W$.
We call $W$ the  Weierstrass model of the elliptic curve (\ref{eq:wei}).
The surface $W$ has a morphism  
$\phi\colon W\to\pr^1=\A^1\cup \A^1$. 
After the minimal resolution of the singular points of $W$, we obtain a nonsingular surface $X$.
This nonsingular projective surface $X$ together with the fibration $\phi\colon X\to \pr^1$ is called the Kodaira-N\'eron model of the elliptic curve $E$ over $\A^1$.

The singular fibers of an elliptic surface are classified by Kodaira. 
The Kodaira type of the elliptic surface $X$ can be computed from the equation (\ref{eq:wei}) using Tate's algorithm.
From the Weierstrass form equation~(\ref{eq:wei}), we can associate two quantities:
$\Delta\coloneqq 4a^3+27b^2,
\
j\coloneqq 4a^3/\Delta$.
Here,
$\Delta$ is the discriminant of the cubic $x^3+a(h)x+b(h)$ and $j$ is the $j$-invariant.
The Kodaira types of the singular fibers are determined  as in  Table~\ref{tab:Tate} by the order of 
$\Delta$ and $j$ which we denote 
${\rm ord}_{v}(\Delta),\ {\rm ord}_{v}(j)$.
\begin{table}[ht]
\centering
\begin{tabular}{|c|c|c|c||c|c|c|c|}\hline
fiber type & Dynkin type & ${\rm ord}_{v}$($\Delta$) &${\rm ord}_{v}(j)$ &fiber type & Dynkin type & ${\rm ord}_{v}$($\Delta$) &${\rm ord}_{v}(j)$\\ \hline
${\rm I_0}$ & - & $0$ & $\geq$0 & ${\rm I_0^{*}}$ & $D_{4}^{(1)}$ & $6$ & $\geq$0\\ \hline
${\rm I}_m$ & $A_{m-1}^{(1)}$ & $m$ & $-m$ &${\rm I}_m^{*}$ & $D_{4+m}^{(1)}$ & $6+m$ & $-m$\\ \hline
${\rm II}$ & - & $2$ & $\geq$0 &${\rm IV^{*}}$ & $E_{6}^{(1)}$ & $8$ & $\geq$0\\ \hline
${\rm III}$ & $A_{1}^{(1)}$ & $3$ & $\geq$0 & ${\rm III^{*}}$ & $E_{7}^{(1)}$ & $9$ & $\geq$0\\ \hline
${\rm IV}$ & $A_{2}^{(1)}$ & $4$ & $\geq$0 & ${\rm II^{*}}$ & $E_{8}^{(1)}$ & $10$ & $\geq$0\\ \hline
\end{tabular}
\caption{Tate's algorithm and Kodaira types}
\label{tab:Tate}
\end{table}

\subsubsection{Elliptic surface associated with the spectral curves}
We introduce the main subject of this article, fibration of spectral curves associated with integrable Lax equations.  
Let us consider a $2n$-dimensional integrable system with a Lax pair.
The spectral curve is parametrized by $n$ functionally independent first integrals $H_1,\dots,H_n$.
\begin{thm}
\label{thm:genus1spectral}
Each elliptic surface whose general fiber is a spectral curve  of the autonomous 2-dimensional Painlev\'e equation has the following singular fiber at $H=\infty$.
 \begin{table}[ht]
 \centering
 \begin{tabular}{|c|c|c|c|c|c|c|c|c|}\hline
 Hamiltonian & $H_{\rm VI}$ & $H_{\rm V}$ &$H_{\rm III(D_6)}$&$H_{\rm III(D_7)}$ &$H_{\rm III(D_8)}$&$H_{\rm IV}$&$H_{\rm II}$&$H_{\rm I}$\\ \hline
Kodaira type & ${\rm I}_0^{*}$ & ${\rm I}_1^{*}$ & ${\rm I}_2^{*}$ & ${\rm I}_3^{*}$& ${\rm I}_4^{*}$& ${\rm IV}^{*}$& ${\rm III}^{*}$& ${\rm II}^{*}$\\ \hline
Dynkin type & $D_4^{(1)}$ & $D_5^{(1)}$ & $D_6^{(1)}$ & $D_7^{(1)}$& $D_8^{(1)}$& $E_6^{(1)}$& $E_7^{(1)}$& $E_8^{(1)}$\\ \hline
 \end{tabular}
 \caption{The singular fiber at $H=\infty$ of spectral curve fibrations of the autonomous 2-dimensional Painlev\'e equations.}
 \end{table}
\end{thm} 
\begin{proof}
First, let us consider the first Painlev\'e equation  
\[
\frac{d^2q}{dt^2}=6q^2+t.
\]
The first Painlev\'e equation has a Lax form 
\begin{align*}
\frac{\partial A}{\partial t}-\delta\frac{\partial B}{\partial x}+[A,B]=0,
\hspace{5mm}
A(x)=
\begin{pmatrix}
	-p & x^2+qx+q^2+\tilde{t}\\
	x-q & p
\end{pmatrix},\hspace{5mm}
B(x)=
\begin{pmatrix}
	0 & x+2q\\
	1 & 0
\end{pmatrix}.
\end{align*}
The spectral curve associated with its autonomous equation is defined by 
$
{\rm det}\left(yI_2-A(x)\right)=0.
$
This is equivalent to the following elliptic curve
\begin{align*}
y^2=x^3+\tilde{t}x+H_{\rm I}, \hspace{3mm}
\hspace{3mm}
H_{\rm I}=p^2-q^3+\tilde{t} q.
\end{align*}
Let us write $H_{\rm I}$ as $h$ for short.
We consider this curve as an elliptic curve $E$ over an affine line $\A^1$.
By  changing variables as
$
\tilde{h}=h^{-1},\ \tilde{x}=h^{-2}x,\ \tilde{y}=h^{-3}y
$,
obtain the ``$\infty$-model'';
\begin{equation}\label{eq:wei4}
\tilde{y}^2=\tilde{x}^3+\tilde{t}\tilde{h}^{4}\tilde{x}+\tilde{h}^5.
\end{equation}
Thus we get the Weierstrass model
 $\varphi\colon W\to \pr^1$.
The Kodaira-N\'eron model 
$\phi\colon X\to \pr^1$ 
of elliptic curve $E$
is obtained from $W$ by the minimal desingularization.
The Kodaira-type of singular fiber at $h=\infty$ can be computed using the equation (\ref{eq:wei4}).
The discriminant of the cubic $\tilde{x}^3+\tilde{t}\tilde{h}^{4}\tilde{x}+\tilde{h}^5$
and  the  $j$-invariant are
\begin{align*}
\Delta=&4\left(\tilde{t}\tilde{h}^{4}\right)^3+27\left(\tilde{h}^5\right)^2
=\tilde{h}^{10}\left(27+4\tilde{t}^3\tilde{h}^2\right),
\\
j=&\frac{4}{\Delta}\left(\tilde{t}\tilde{h}^4\right)^3
=\frac{4\tilde{t}^3\tilde{h}^{12}}{\tilde{h}^{10}(27+2\tilde{t}^3\tilde{h}^2)}
=\tilde{h}^2\frac{4\tilde{t}^3}{27+4\tilde{t}^3\tilde{h}^2}.
\end{align*}
Therefore, the order of zero of $\Delta$ and $j$ at $\tilde{h}=0$ is
${\rm ord}_{\infty}(\Delta)=10, {\rm ord}_{\infty}(j)=2$.
Using  Tate's algorithm, we find that the elliptic surface $X\to \pr^1$ has the singular fiber of type ${\rm II}^{*}$.
In  Dynkin's notation, this fiber is of type ${\rm E}_{8}^{(1)}$.
Let us express the other two zeros of the discriminant $\Delta$ by $h_{+}$ and $h_{-}$.
Since $h_{+}$ and $h_{-}$ are both simple zeros of $\Delta$, the Kodaira type of the fibers at $h_{+}$ and $h_{-}$ are ${\rm I_{1}}$, from Tate's algorithm.
\begin{multicols}{2}
\begin{figure}[H]
	\includegraphics[scale=0.4]{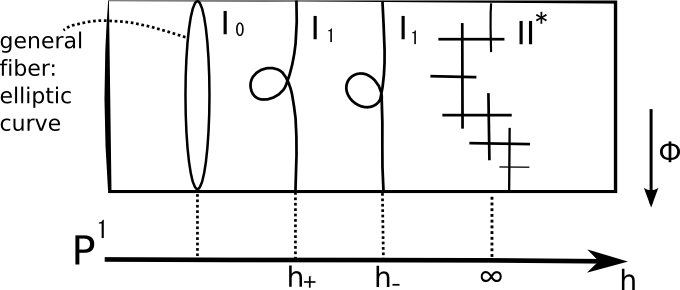}	
	\caption{The elliptic surface associated to  the spectral curves of the autonomous $P_{\rm I}$.	}
\end{figure}	

\columnbreak
\begin{align*}
\begin{xy}
\ar@{-} 
(0,0) *+[Fo]{2}="J";
(10,0) *+[Fo]{4}="H" \ar@{-} "H";
(20,0)  *+[Fo]{6}="C" \ar@{-} "C";
(20,10) *+[Fo]{3}="A" \ar@{-} "A";
(30,0) *+[Fo]{5}="D" ;
(40,0) *+[Fo]{4}="E" ;
(50,0) *+[Fo]{3}="F" ;
(60,0) *+[Fo]{2}="I" ;
(70,0) *+[Fo]{1}="K" ;
\ar@{-} "C";"D";
\ar@{-} "D";"E";
\ar@{-} "E";"F";
\ar@{-} "F";"I";
\ar@{-} "I";"K";
\end{xy}
\end{align*}
The dual graph of the singular fiber of the Kodaira type ${\rm II}^{*}$ (Dynkin type $E_8^{(1)}$).
The numbers in circles denote the multiplicities of components in the reducible fibers.
\end{multicols}
The spectral curves associated to other autonomous 2-dimensional Painlev\'e equations  are also curves of genus one for the general values of the Hamiltonians.
It is well known that  curves with genus one can always be transformed into the Weierstrass normal form.
With the aid of computer programs,
 we can transform the spectral curves of autonomous 2-dimensional Painlev\'e equations into the Weierstrass normal form\footnote{Magma, Sage and Maple serve this purpose. 
 Magma even calculates Kodaira types from the equations.}.
Once the spectral curves are in the Weierstrass form, 
we  can construct the Weierstrass model.
After the minimal desingularizations, we obtain the elliptic surfaces.  
We  apply Tate's algorithm to find the Kodaira types of the singular fibers of these elliptic surfaces. 
\end{proof}
\begin{rem}
For the special values of the parameters and the constant $\tilde{t}$ of the equations, the order of zeros or poles of $\Delta$ and $j$ changes so that the Kodaira type changes.
This corresponds to the situations when the equations have the special solutions.
In this paper, we concentrate on the situations when the parameters are generic. 
\end{rem}
\subsubsection{The Liouville tori and elliptic surface}
For the autonomous 2-dimensional Painlev\'e-type equations, the Liouville tori are elliptic curves.
Therefore, elliptic surfaces are naturally associated to these integrable systems as Hamiltonian fibrations.
We think of the time variable $\tilde{t}$ as a constant.
\begin{thm}
Each elliptic surface associated to the autonomous 2-dimensional Painlev\'e equation as the Hamiltonian fibration has the following singular fiber at $h=\infty$.
 \begin{table}[ht]
 	\centering
 	\begin{tabular}{|c|c|c|c|c|c|c|c|c|}\hline
 		Hamiltonian & $H_{\rm VI}$ & $H_{\rm V}$ &$H_{\rm III(D_6)}$&$H_{\rm III(D_7)}$ &$H_{\rm III(D_8)}$&$H_{\rm IV}$&$H_{\rm II}$&$H_{\rm I}$\\ \hline
 		Kodaira type & ${\rm I}_0^{*}$ & ${\rm I}_1^{*}$ & ${\rm I}_2^{*}$ & ${\rm I}_3^{*}$& ${\rm I}_4^{*}$& ${\rm IV}^{*}$& ${\rm III}^{*}$& ${\rm II}^{*}$\\ \hline
 		Dynkin type & $D_4^{(1)}$ & $D_5^{(1)}$ & $D_6^{(1)}$ & $D_7^{(1)}$& $D_8^{(1)}$& $E_6^{(1)}$& $E_7^{(1)}$& $E_8^{(1)}$\\ \hline
 	\end{tabular}
 	\label{thm:genus1Liouville}
 \end{table}
\end{thm} 
\begin{proof}
Let us first consider the easiest case: the first Painlev\'e equation.
We write $h=H_{\rm I}$ for short.
The Hamiltonian of the first Painlev\'e equation is
$
h=p^2-(q^3+\tilde{t} q).
$
We view it as an elliptic curve over $\A^1$:
\begin{align*}
\left\{(q,p,h)\in \A^2_{(q,p)}\times\A^1_h\relmiddle| p^2=q^3+\tilde{t}q+h \right\}
\to
\A^1_h.
\end{align*}
As in the case of the spectral curve fibration, we can construct the Kodaira-N\'eron model of the elliptic surface from the equation.
Replacing $ \tilde{q}=q/h^2,\ \tilde{p}=p/h^3,\ \tilde{h}=1/h$,
we obtain the $\infty$-model:
\begin{align*}
\tilde{p}^2=
\tilde{q}^3+\tilde{t}\tilde{h}^4\tilde{q}+\tilde{h}^5.
\end{align*}
After  compactification and the  minimal desingularization, we obtain a regular elliptic surface whose general fiber at $h$ is the elliptic curve defined by $ p^2=q^3+\tilde{t}q+h$.
The discriminant and the $j$-invariant are:
\begin{align*}
\Delta=
4\left(\tilde{t}\tilde{h}^4\right)^3+27\left(\tilde{h}^5\right)^2
=\tilde{h}^{10}(27+4\tilde{t}^3\tilde{h}^2),
\hspace{3mm}
j=
\frac{4(\tilde{t}\tilde{h}^4)^3}{4\tilde{h}^{10}(27+4\tilde{t}^3\tilde{h}^2)}
=\tilde{h}^{2}\frac{\tilde{t}^3}{27+4\tilde{t}^3\tilde{h}^2}.
\end{align*}
Thus the order of zero of $\Delta$ and $j$ at $\tilde{h}=0$ are
${\rm ord}_{\infty}(\Delta)=10, {\rm ord}_{\infty}(j)=2$.
It follows from Tate's algorithm that the singular fiber at $h=\infty$ of the autonomous $P_{\rm I}$-Hamiltonian fibration is of type ${\rm II}^{*}$, or $E_8^{(1)}$ in the Dynkin's notation.

We use computer programs to transform the other Hamiltonians into the Weierstrass normal form.
The rest of the proof is similar to the case of $H_{\rm I}$.
\end{proof}
The agreement of the singular fibers at $h=\infty$ of the spectral curve fibrations and the Liouville torus fibrations is not a coincidence.
Liouville tori are related to the Jacobian varieties of the spectral curves, and taking the Jacobians are isomorphism in genus 1 cases by Abel's theorem.
It might be natural to study the fibration of the Liouville tori, but we have to deal with families of 2-dimensional Abelian varieties to study the autonomous 4-dimensional  Painlev\'e-type equations. 
Therefore studying the degeneration of the Liouville tori become harder compared to the cases of the 2-dimensional Painlev\'e equations.
On the other hand, we only need to deal with genus 2 curves to study spectral curve fibrations of the 4-dimensional autonomous Painlev\'e type equations.
Thus, studying the degeneration of spectral curves is the main object of  this paper.

Using blowing-up process, Okamoto resolved the singularities of 2-dimensional Painlev\'e differential equations and constructed the ``spaces of initial conditions''~\cite{MR614694}.
While he deals with singularities of the systems of differential equations, we deal with spectral curves or Hamiltonians themselves for autonomous cases.

A space of initial conditions can be characterized by a pair $(X, D)$ of
a rational surface $X$ and the anti-canonical divisor $D$ of $X$. 
Each irreducible
component of $D$ is a rational curve and, in the case of the Painlev\'e equations,
is called as a vertical leaf~\cite{MR1882403}.
 The intersection diagram of $D$ is given by
that of the certain root lattice listed above.
\begin{rem}
The spaces of initial conditions are also considered for  several  cases of 4-dimensional Garnier systems and Noumi-Yamada systems~\cite{MR1215998,MR1778322,MR2276182,MR2117253}. 
\end{rem}

If we restrict our attention to the autonomous cases, the geometrical studies are much simpler.
The autonomous 2-dimensional Painlev\'e equations constructed from the spaces of initial conditions were studied by Sakai~\cite{MR3077699}.

\subsection{Degeneration of genus 2 curves and Liu's algorithm}
We  apply a similar method as in the previous subsection to the 40 types of the autonomous  4-dimensional Painlev\'e-type equations.
While genus of spectral curves of the autonomous 2-dimensional Painlev\'e equations is $1$, genus of the autonomous 4-dimensional Painlev\'e-type equations is $2$.
As we  use Tate's algorithm to determine the fibers  of the N\'eron models for the genus 1 curves, we use Liu's algorithm for the genus 2 curves.
While the spectral curves of the autonomous 2-dimensional Painlev\'e equations are 1-parameter families parameterized by the Hamiltonian, those of the 4-dimensional Painlev\'e-type equations are 2-parameters families parametrized by two first integrals.

The spectral curve is now 2-parameter family of genus two curves
\begin{align*}
F(w,x,y,h_1,h_2)=0,
\end{align*}
where $h_1$ and $h_2$ are the functionally independent conserved quantities of the system and $F(w,x,y,h_1,h_2)\in\C[w,x,y,h_1,h_2]$.
Let us consider a generic line in $\operatorname{Spec}(\C[h_1,h_2])$
\begin{align*}
ah_1+bh_2=c,
\end{align*}
where $a, b, c\in\C$ are generic. 
Upon replacing $\tilde{a}=-a/b$, $\tilde{b}=c/b$,
\begin{align*}
h_2=\tilde{a}h_1+\tilde{b}
\end{align*}
we obtain a one-parameter family of spectral curves
\begin{align*}
\tilde{F}(w,x,y,h_1)
\coloneqq
F(w,x,y,h_1,\tilde{a}h_1+\tilde{b})=0.
\end{align*}
We study degeneration of this one-parameter family of spectral curves assuming $\tilde{a}$ and $\tilde{b}$ are generic.

\subsubsection{Liu's algorithm}
We summarize studies on the degeneration of genus two curves. 
The numerical classification of the fibers in pencils of genus 2 curves are given by Ogg \cite{MR0201437} and Iitaka \cite{Iitaka}. 
Namikawa and Ueno~\cite{MR0369362}  completed the geometrical classification of such fibers (and added a few missing types in \cite{MR0201437} and \cite{Iitaka}). 
There are 120 types in Namikawa-Ueno's classification, while there are only 10 types in Kodaira's classification of the fibers in pencils of genus 1 curves.\footnote{Kodaira type ${\rm I}_n~(n\geq 1)$ and ${\rm I}_n^*~(n\geq 1)$ are  counted as 1 type, respectively.}
Liu gave an algorithm similar to Tate's algorithm for genus 2 curves~\cite{MR1202389,MR1285783}.
Using Liu's algorithm, we can determine the Namikawa-Ueno type of singular fibers from explicit equations of pencils of genus 2 hyperelliptic curves in the Weierstrass form.

\begin{table}[H]	
	\centering	
	\begin{tabular}{|c|c|c|c|}
		\hline 
		& genus of 
		& types of singular
		& algorithm to determine
		\\
		& spectral curve 
		&  fibers in pencils
		&  types of fibers
		\\
		\hline
		\hline
		2-dim. Painlev\'e 
		& 1 
		& Kodaira
		& Tate's algorithm
		\\
		\hline
		4-dim. Painlev\'e 
		& 2 
		& Namikawa-Ueno
		& Liu's algorithm
		\\
		\hline
	\end{tabular}
\end{table}	

\subsection{Generic degeneration of the spectral curves of of the autonomous 4-dimensional Painlev\'e-type equations}
Let us state the main theorem of this paper.
\begin{thm}\label{thm:table}
The spectral curves of the autonomous 4-dimensional Painlev\'e-type equations have the following types of generic degenerations as in Table~\ref{table:generic}.
\end{thm}
\begin{rem}
Let us explain the notations used in Table~\ref{table:generic}.
The Hamiltonians are the Hamiltonians of the 4-dimensional Painlev\'e-type equations.
The explicit forms of the non-autonomous counterparts can be found in \cite{sakai_imd,MR3087954,2016arXiv160803927K, 2016arXiv160905263K, MR3740334,MR1042827,MR2547101}.
The spectral types indicate the type of  corresponding linear equations.
Such notations are explained in Subsection~\ref{subsec:imd} and \ref{ap:linear}. %
``Namikawa-Ueno type'' means the  types of degeneration of genus 2 curves in Namikawa-Ueno~\cite{MR0369362}.
When the fiber contains components expressible by the Kodaira-type, we also write its Dynkin's name in the column  noted ``Dynkin''.
The column named ``stable'' tells us the type of the stable model~\cite{MR1202389}.
The ``$\Phi$'' indicates the group of connected components 
 of the N\'eron model of the Jacobian $J(C)$.
The symbol $(n)$ means the cyclic group with $n$ elements.\footnote{When $n=0$, $(n)$ is the trivial group.}
We also write Ogg's type written in ``On pencils of curves of genus two''~\cite{MR0201437}.
Ogg uses the notation ``Kod'' to express Kodaira-type and do not distinguish them, while Namikawa and Ueno  does.
Ogg's type might be helpful to see the rough classification.
For example, all 8 types of matrix Painlev\'e equations have the same Ogg's type 14.
The column ``monodromy'' means 5 monodromy types in Namikawa-Ueno~\cite{MR0369362}.
Elliptic types are those with finite degrees of monodromy, while parabolic types have infinite degrees.
Elliptic[1] are those with stable model ``I'' in Liu's notation.
We abbreviate as ``ell[1]''.
We summarize such correspondences in the following table.
\begin{table}[h]
\centering
\begin{tabular}{|c|c|c|}
\hline
Type  & Order of monodromy & stable type (Liu's notation~\cite{MR1202389})
 \\ \hline \hline
 elliptic[1] & finite & I
 \\ \hline
 elliptic[2] & finite & V
  \\ \hline
 parabolic[3] & infinite & II, VI
  \\ \hline
  parabolic[4] & infinite & III,VII
   \\ \hline
    parabolic[5] & infinite & IV
    \\ \hline
\end{tabular}
\caption{Namikawa-Ueno's elliptic and parabolic types and Liu's stable model types}
\end{table}
The column named ``page'' indicate the page number of Namikawa-Ueno's paper where some data of the corresponding type can be found.
\end{rem}

\begin{table}[H]	
	\centering	
	\begin{tabular}{|c|c|}
		\hline 
		Hamiltonian
		&
		the Hamiltonian of the Painlev\'e-type equation\cite{sakai_imd,kns, 2016arXiv160803927K,2016arXiv160905263K,MR3740334}
		\\
		\hline
		spectral type
		&
		the spectral type of the corresponding linear equation~\cite{BB10865098,kns, 2016arXiv160803927K}, Appendix~\ref{ap:linear}
		\\
		\hline
		monodromy
		&
		5 types of monodromy (elliptic[1],[2],Parabolic[3],[4],[5]) as in Namikawa-Ueno~\cite{MR0369362}
		\\
		\hline
		Namikawa-Ueno 
		&
		the type of fiber in the minimal model following the notation in Namikawa-Ueno~\cite{MR0369362,MR1285783}
		\\
		\hline
		Dynkin
		&
		Dynkin type (when the fiber contains Kodaira-type component)
		\\
		\hline
		stable
		&
		the type of stable model of the fiber~\cite{MR1202389}
		\\
		\hline
		$\Phi$
		&
		the group of connected components of the N\'eron model of the Jacobian $J(C)$~\cite{MR1285783}
		\\
		\hline
		Ogg
		&
		the type of fiber in the minimal model following the notation in Ogg~\cite{MR0201437}
		\\
		\hline
		page
		&
		the page number of the fiber in Namikawa-Ueno's paper~\cite{MR0369362}
		\\
		\hline
	\end{tabular}
\end{table}	

\begin{proof}[Proof of Theorem~\ref{thm:table}]
Let us take Gar$\frac{9}{2}$, the most degenerated Garnier system, to demonstrate our computation.
The Lax pair is given by
\begin{align*}
\frac{dA(x)}{dt_i}+[A(x),B_i(x)]=0,
\hspace{3mm}
i=1,2
\end{align*}
where
\begin{align*}
A(x)\ =&A_0x^3+A_1x^2+A_2x+A_3,
\\
B_1(x)=&A_0x^2+A_1x+B_{10}
=\frac{A(x)}{x}+C_1-\frac{A_3}{x},
\\
B_2(x)=&-A_0x+B_{20},
\end{align*}
for
\begin{align*}
&A_0=
\begin{pmatrix}
0 & 1
\\
0 & 0
\end{pmatrix},
\
A_1=
\begin{pmatrix}
0 & p_1
\\
1 & 0
\end{pmatrix},
A_2=
\begin{pmatrix}
q_2 & p_1^2+p_2+2\tilde{t}_1
\\
-p_1 & -q_2
\end{pmatrix},
\\
&A_3=
\begin{pmatrix}
q_1-p_1q_2\ & \ p_1^3+2p_1p_2-q_2^2+\tilde{t}_1p_1-\tilde{t}_2
\\
-p_2+\tilde{t}_1 & -q_1+p_1q_2
\end{pmatrix},
\\
&B_{10}=
\begin{pmatrix}
q_2 & p_1^2+2p_2+\tilde{t}_1
\\
-p_1 & -q_2
\end{pmatrix},
B_{20}=
\begin{pmatrix}
0 & -2p_2
\\
-1 & 0
\end{pmatrix},
C_{1}=
\begin{pmatrix}
0 & p_2-\tilde{t}_1
\\
0 & 0
\end{pmatrix}.
\end{align*}
	The characteristic polynomial 
\begin{align*}
\det(yI_2-A(x))=0
\end{align*}	
	 is expressed as
\begin{align*}
y^2=x^5+3\tilde{t}_2x^3-\tilde{t}_1x^2+(2\tilde{t}_2^2-h_1)x+h_2-\tilde{t}_1\tilde{t}_2,
\end{align*}
where $h_1=H_{{\rm Gar}, \tilde{t}_1}^{9/2}$, $h_2=H_{{\rm Gar}, \tilde{t}_2}^{9/2}$.
Note that it is already in the Weierstrass form.
We consider the degeneration along a line $h_2=ah_1+b$, where $a$ and $b$ are generic constants.
\begin{align*}
y^2=x^5+3\tilde{t}_2x^3-\tilde{t}_1x^2+(2\tilde{t}_2^2-h_1)x+ah_1+b-\tilde{t}_1\tilde{t}_2,
\end{align*}
In order to see the degeneration at $h_1=\infty$, we introduce
$\tilde{x}=x/h_1,\ \tilde{y}=y/h_1^3,\ \tilde{h}=1/h_1$.
\[
\tilde{y}^2=
\tilde{h} x^5
+3\tilde{t}_2  \tilde{h}^3 x^3
-\tilde{t}_1\tilde{h}^4 x^2
-\tilde{h}^4 x
+2 \tilde{t}_2^2 \tilde{h}^5 x
+\tilde{h}^5(a+b \tilde{h}-\tilde{t}_1\tilde{t}_2\tilde{h})
.
\]
The Igusa invariants of the quintic can be calculated as follows.
\begin{align*}
J_2 &{} =-5\tilde{h}^5+\frac{67}{4}\tilde{t}_2^2\tilde{h}^6 ,\quad
J_4=\frac{15}{8}\tilde{h}^{10}+O(\tilde{h}^{11}),\quad
J_6  =\frac{5}{16}\tilde{h}^{15}+O(\tilde{h}^{16}) ,\\
J_8&{}=-\frac{325}{256}\tilde{h}^{20}+O(\tilde{h}^{21}),
\quad
J_{10} 
=-\frac{1}{16}\tilde{h}^{25}+O(\tilde{h}^{26}) 
.
\end{align*}
Since $5\cdot\operatorname{ord}_{\infty}J_{2i}-i\cdot \operatorname{ord}_{\infty}J_{10}=0$ for $i\leq5$, the stable model has smooth fiber  (type(I) in Liu's notation of stable curves) at $\tilde{h}=0$.
 With further computation,
 we find that generic degeneration is  of type ${\rm VII}^{*}$ in  Namikawa-Ueno's notation. 
This type is type 22 in Ogg's notation~\cite{MR0201437}.

${\rm VII^{*}}\colon H_{{\rm Gar},\tilde{t_1}}^{\frac{9}{2}}$
\[
\begin{xy}
\ar@{-}  (-10,0) *+[Fo]{1};
(0,0) *+[Fo]{2B}="J" \ar@{-} "J";
(10,0) *+[Fo]{5}="H" \ar@{-} "H";
(20,0)  *+[Fo]{8}="C" \ar@{-} "C";
(20,10) *+[Fo]{4}="A" \ar@{-} "A";
(30,0) *+[Fo]{7}="D" ;
(40,0) *+[Fo]{6}="E" ;
(50,0) *+[Fo]{5}="F" ;
(60,0) *+[Fo]{4}="I" ;
(70,0) *+[Fo]{3}="K" ;
(80,0) *+[Fo]{2}="L" ;
(90,0) *+[Fo]{1}="M" ;
\ar@{-} "C";"D";
\ar@{-} "D";"E";
\ar@{-} "E";"F";
\ar@{-} "F";"I";
\ar@{-} "I";"K";
\ar@{-} "K";"L";
\end{xy}
\]
The numbers in circles denote the multiplicities of components in the reducible fibers.
All curves are (-2)-curves except the one expressed as ``$B$'',  which is a (-3)-curve.

We can transform the spectral curves of the other autonomous 4-dimensional Painlev\'e equations into the Weierstrass normal form.\footnote{We used Maple's command ``Weierstrassform'' implemented by van Hoeij.}
We restrict the 2-parameter families of genus 2 curves to a generic line on the base and   apply Liu's algorithm to find the generic degeneration.
\end{proof}

\begin{rem}
In this paper, we proposed a possible clue to characterize the integrable systems studying the degeneration of spectral curves.
The Namikawa-Ueno types and the monodromy matrices of the generic singular fibers are given.
Our plan in the future work is to understand the phase spaces as the relative compactified Jacobian of the spectral curve fibration by studying the discriminant locus of the base space and combine with the  results on monodromies of the singular fibers. 
\end{rem}

\newpage

\thispagestyle{empty}
\begin{table}[h]
\centering
{
   \small
  \begin{tabular}{|l|c|c|c|c|c|c|c|c|}\hline
   Hamiltonian 
   & spectral type 
   & monodromy
   & Namikawa-Ueno
   & Dynkin 
   & stable 
   &    $\Phi$ 
   & Ogg
   & page \\ \hline \hline
      $H_{{\rm Gar}}^{1+1+1+1+1}$ 
      & 11,11,11,11,11 
      & ell[1] 
      & ${\rm I^{*}_{0-0-0}}$ 
      &-
      & II
      & $(2)^4$
      & 33      
      & p.155
      \\ \hline
      $H_{{\rm Gar}}^{2+1+1+1}$
      &(1)(1),11,11,11  
      & par[3]
      & ${\rm I^{*}_{1-0-0}}$
      &-
      & II
      & $(4)\times (2)^2$
      & 33
      & p.171
      \\ \hline
       $H_{{\rm Gar}}^{3/2+1+1+1}$
      &(2)(1),(1)(1)(1)
    & par[3]
    & ${\rm I^{*}_{2-0-0}}$
    & -
    & II
    & $ (2)^4$
    & 33
    & p.171
      \\ \hline
      $H_{{\rm Gar}}^{2+2+1}$ 
    &(1)(1),(1)(1),11  
    & par[4]
    & ${\rm I^{*}_{1-1-0}}$
    &  - 
    & III 
    & $(4)^2$
    &  33
    & p.180
    \\ \hline
    $H_{{\rm Gar}}^{3/2+2+1}$ 
    & $(1)_2,(1)(1),11$ 
    & par[4]
    & ${\rm I^{*}_{1-2-0}}$
    & -
    & III
    & $ (4)\times(2)^2$
    & 33
    & p.180
    \\ \hline
    $H_{{\rm Gar}}^{3/2+3/2+1}$ 
    & $(1)_2,(1)_2,11$ 
    & par[4]
    &${\rm I^{*}_{2-2-0}}$  
    &-
    & III
    & $(2)^4$
    & 33
    & p.180
    \\ \hline
       $H_{{\rm Gar}}^{3+1+1}$ 
        &((1))((1)),11,11 
        & ell[2]
        & ${\rm I_{0}^{*}-IV^{*}-(-1) }$ 
        &$ D_{4}-E_{6}-(-1)$
        & V 
        & $(6)\times (2)$
        & 29a
        & p.161
        \\ \hline
      $H_{{\rm Gar}}^{5/2+1+1}$ 
        &(((1)(1)))(((1)))
        & ell[2]
        &${\rm I_0^{*}-III^{*}-(-1)}$
        & $ D_{4}-E_{7}-(-1)$
        & V
        & $(2)^3$
        & 29a
        & p.162
        \\ \hline  
      $H_{{\rm Gar}}^{4+1}$ 
    &(((1)))(((1))),11 
    & par[3] 
    & ${\rm III^{*}-II^{*}_{0}}$ 
    & $E_{7}-{\rm II}^{*}_{0}$
    & II
    & (8)
    & 23
    & p.178
    \\ \hline
    $H_{{\rm Gar}}^{7/2+1}$ 
          & $(((((1)))))_2,11$
          & par[3]
          & ${\rm II^{*}-{\rm II}_{0}^{*}}$
          & $ E_8-{\rm II}_0^{*}$
          & II
          & $(4)$
          & 25
          & p.176
           \\ \hline
      $H_{{\rm Gar}}^{3+2}$
       &((1))((1)),(1)(1)
       & par[3]
       & ${\rm IV^{*}-I^{*}_1-{(-1)}}$
       & $E_6-D_5-(-1)$ 
       & VI
       & (12) 
       & 29a
       & p.175
       \\ \hline
    $H_{{\rm Gar}}^{5/2+2}$ 
    &$(((1)))_2,(1)(1)$
    & par[3]
    &  ${\rm III^*}-{\rm I^*_{1}}-(-1)$
    & $E_7-D_5-(-1)$
    & VI
    & $(4)\times(2)$
    & 29a
    & p.177
    \\ \hline  
        $H_{{\rm Gar}}^{3/2+3}$
      & $(1)_2,((1))((1))$
       & par[3]
       & ${\rm IV^{*}-I^{*}_2-{(-1)}}$
       & $E_6-D_6-(-1)$ 
       & VI
       & $(6)\times(2)$ 
       & 29a
       & p.175
      \\ \hline
      $H_{{\rm Gar}}^{5/2+3/2}$
      & $(((1)))_2,(1)_2$
      & para[3]
      & ${\rm III^{*}-I^{*}_2-(-1)}$ 
      &$ E_7-D_6-(-1)$
      & VI 
      &$(2)^3$
     & 29a 
      & p.177
      \\ \hline
   $H_{{\rm Gar}}^{5}$
       &((((1))))((((1))))
       & ell[1] 
       & ${\rm IX-3}$
       &-
       & I
       & (5)
       & 21 
       & p.157
       \\ \hline
      $H_{{\rm Gar}}^{9/2}$ 
      & $(((((((1)))))))_2$ 
       & ell[1]
      & ${\rm VII^{*}}$ 
      &-
      & I 
      &(2)
      & 22 
      & p.156
      \\ \hline
      $H_{\rm FS}^{A_5}$ 
      & 21,21,111,111
      & par[3]
      & ${\rm II_{3-0}}$
      &-
      & II
      &(12)
      & 41
      & p.171
      \\ \hline
     $H_{\rm FS}^{A_4}$ 
     &(11)(1),21,111 
     &par[5]
     & ${\rm II_{3-1}}$ 
     &- 
     &IV  
     &(13)
     &41    
     &p.183
     \\ \hline
        $H_{FS}^{A_3}$
      & $(1)_2,21,111$ 
      &par[5]
      & ${\rm II_{3-2}}$ 
      &- 
      &IV  
      &(14)
      &41    
      &p.183
      \\ \hline
    $H_{\rm Suz}^{\frac{3}{2}+2}$
      & $(11)(1),(1)_{2}1$
      & par[5]
      &${\rm II_{3-3}} $
      &- 
      & IV 
      & (15)
      & 41
      & p.183
      \\ \hline
   $H_{\rm KFS}^{\frac{3}{2}+\frac{3}{2}}$
      & $(1)_3,(11)(1)$
      & par[5]
      & ${\rm II_{3-4}}$
      &-
      & IV
      & (16)
      & 41 
      & p.183
      \\ \hline
   $H_{\rm KFS}^{\frac{4}{3}+\frac{3}{2}}$
     & $(1)_3,(1)_21$
     & par[5]
     & ${\rm II_{3-5}}$
     &  - 
     & IV 
     &(17)
     & 41
     & p.183
      \\ \hline
    $H_{\rm KFS}^{\frac{4}{3}+\frac{4}{3}}$
    & $(1)_3,(1)_3$
    & par[5]
    & ${\rm II_{3-6}}$
    & - 
    & IV 
    &(18) 
    & 41
    & p.183
     \\ \hline
     $H_{\rm NY}^{A_5}$ 
     & (2)(1),111,111 
     & par[4]
     & ${\rm II_{3-1}}$
     &- 
     & III 
     & $(6)\times (2)$ 
     & 41a  
     & p.182
     \\ \hline
     $H_{\rm NY}^{A_4}$ 
     &((11))((1)),111
     & par[3]
     & ${\rm IV^{*}-II_{3}}$ 
     &$ E_{6}-{\rm II}_{3}$
     & II
     &(10) 
     & 41b  
     & p.175
     \\ \hline
    $H_{\rm Ss}^{D_6}$
     & 31,22,22,1111 
     & par[3]
     & ${\rm I_2-I_0^{*}-0}$
     &$ A_1-D_4-0$
     & VI
     & $ (2)^3$
     & 2    
     & p.171
     \\ \hline
$H_{\rm Ss}^{D_5}$
    &(111)(1),22,22
    & par[4]
    & ${\rm I_{2}-I_{1}^{*}-0}$
    & $ A_{1}-D_{5}-0$
    & VII 
    & $(4)\times (2)$
    & 2 
    & p.180
    \\ \hline
    $H_{\rm Ss}^{D_4}$
     & (2)(2),(111)(1) 
     & par[4]
     & ${\rm I_{2}-I_{2}^{*}-0}$
     &$ A_{1}-D_{6}-0$
     & VII
     & $ (2)^4$
     & 2
     & p.180
     \\ \hline
   $H_{\rm KSs}^{\frac{3}{2}+2}$ 
   & $(1)_{2}11,(2)(2)$ 
   &par[4] 
   & ${\rm I_{2}-I_{3}^{*}-0}$  
   &$ A_{1}-D_{7}-0$
   & VII  
   & $(4)\times (2)$
   & 2
   & p.180
   \\ \hline
$H_{\rm KSs}^{\frac{4}{3}+2}$ 
     & $(1)_31,(2)(2)$ 
     & par[4] 
     &${\rm I_2-I^{*}_{4}-0}$  
     &$ A_1-D_{8}-0$
     & VII   
     & $ (2)^3$
     & 2 
     & p.180
    \\ \hline
    $H_{\rm KSs}^{\frac{5}{4}+2}$
    & $(1)_4,(2)(2)$
    & par[4] 
    & ${\rm I_2-I^{*}_{5}-0}$  
    & $ A_1-D_{9}-0$
    & VII  
    & $(4)\times (2)$
    & 2 
    & p.180
    \\ \hline
     $H_{\rm KSs}^{\frac{3}{2}+\frac{5}{4}}$ 
     & $(1)_4,(2)_2$ 
     & par[4] 
     &${\rm I_2-I^{*}_{6}-0}$  
     &$ A_1-D_{10}-0$
     & VII   
     & $ (2)^3$
     & 2 
     & p.180
    \\ \hline
    $H_{\rm VI}^{\rm Mat}$ 
    & 22,22,22,211 
    & ell[2]
    &${\rm I_{0}-I_{0}^{*}-1}$ 
    & $ {\rm I_{0}}-D_{4}-1$
    & V
    & $(2)^2$
    & 14
    & p.159
    \\ \hline
   $H_{\rm V}^{\rm Mat}$ 
    & (2)(11),22,22
    & par[3]
    &  ${\rm I_{0}-I_{1}^{*}-1}$ 
    & $ {\rm I_{0}}-D_{5}-1$
    & VI
    & $(4)$
    & 14  
    & p.170
    \\ \hline
   $H_{\rm III(D_{6})}^{\rm Mat}$ 
   & (2)(2),(2)(11)
   & par[3]
   &${\rm I_{0}-I_{2}^{*}-1}$ 
   & $ {\rm I_{0}}-D_{6}-1$
   & VI
   & $(2)^2$
   & 14
   & p.170
   \\  \hline
   $H_{\rm III(D_{7})}^{\rm Mat}$ 
   & $(2)(2),(11)_2$ 
   & par[3]
   & ${\rm I_{0}-I_{3}^{*}-1}$
   &${\rm I_{0}}-D_{7}-1$
   & VI
   & $(4)$
   & 14
   & p.170
   \\ \hline
   $H_{\rm III(D_{8})}^{\rm Mat}$
    & $(2)_2,(11)_2$ 
    & par[3]
    & ${\rm I_{0}-I_{4}^{*}-1}$
    & ${\rm I_{0}}-D_{8}-1$
    & VI
    & $(2)^2$
    & 14
    & p.170
    \\ \hline
   $H_{\rm IV}^{\rm Mat}$
    & ((2))((11)),22 
    & ell[2]
    & ${\rm I_{0}-IV^{*}-1}$ 
    &${\rm I_{0}}-E_{6}-1$
    & V
    &  (3)
    & 14
    & p.160
    \\ \hline
   $H_{\rm II}^{\rm Mat}$
    &(((2)))(((11)))
    & ell[2]
    &${\rm I_0-III^{*}-1}$ 
    & ${\rm I_{0}}-E_{7}-1$
    & V
    &(2)
    & 14
    & p.162
    \\ \hline
  $H_{\rm I}^{\rm Mat}$ 
   & $(((((11)))))_2$
   & ell[2]
   & ${\rm I_{0}-II^{*}-1}$
   &${\rm I_{0}}-E_{8}-1$
   & V
   & ${\rm 0}$
   & 14
   & p.160
   \\ \hline
  \end{tabular}
}
\caption{The generic degeneration of the spectral curves of the  autonomous 4-dimensional Painlev\'e-type equations}
\label{table:generic}
\end{table}

\clearpage

\clearpage

\appendix
\section*{Appendix}
\setcounter{section}{1}
\subsection{Conserved quantities}\label{ap:con}
The autonomous 4-dimensional Painlev\'e-type equations have two functionally independent first integrals.
In this subsection, we list these first integrals for the ramified equations.\footnote{We do not write first integrals of the Garnier equations here, since the first integrals are just autonomous limit of two Hamiltonians.
Such Hamiltonians are listed by Kimura~\cite{MR1042827}, Kawamuko~\cite{MR2547101} and Kawakami~\cite{MR3740334}.}
One of the reason is that the other first integrals than the Hamiltonians  have long expressions.
Writing them for ``less-degenerated'' systems take huge spaces.
But they are easily computable from data in the previous paper~\cite{kns}. 
We only give first integrals for autonomous version of equations in Kawakami~\cite{2016arXiv160803927K, 2016arXiv160905263K, MR3740334}.
We list Hamiltonians $H$'s with $\delta$,\footnote{Hamiltonians for the case $\delta=0$  is the first integral.}
 and the other invariants $G$'s.

There are 5 ramified cases from the degeneration of $A_5$ Fuji-Suzuki system.\footnote{Although we obtain the Garnier equations of ramified types from certain degenerations of $A_5$-type Fuji-Suzuki system, we  exclude the Garnier systems.}
{\allowdisplaybreaks
\begin{align*}
H_{\rm FS}^{A_3}
= {}& 
H_{\rm III}(D_6)(-\theta_2^\infty,\delta+\theta_1^0+\theta_1;\tilde{t};q_1,p_1)
+H_{\rm III}(D_6)(\theta_2^0-\theta_1^0,\theta_2^0-\theta_1^0-\theta^1;\tilde{t};q_2,p_2)
-\frac{1}{\tilde{t}}p_1p_2(q_1q_2+\tilde{t}),
\\
G_{\rm FS}^{A_3}
= {}&
\theta_2^0 \tilde{t} 
\left(
H_{\rm III}(D_6)(-\theta_2^\infty,\theta_1+\theta_1^0;\tilde{t};q_1,p_1)
-p_1p_2
\right)
\\
&
+\left(q_1 q_2-\tilde{t}\right) \left(\theta_2^{\infty} p_2-\left(\theta_1^0-\theta_2^0\right) p_1
   +p_1 p_2 \left(\theta_1+\theta_1^0-\theta_2^0+\left(p_1-1\right)
   q_1-\left(p_2-1\right) q_2\right)\right),
\\
H_{\rm KFS}^{\frac{3}{2}+2}
= {}& 
H_{\rm III}(D_7)(-\theta_1^0;\tilde{t};q_1,p_1)
+H_{\rm III}(D_7)(\theta_2^0-\theta_1^0;\tilde{t};q_2,p_2)
+\frac{1}{\tilde{t}}
\left(p_2q_1(p_1(q_1+q_2)+\theta^{\infty}_{2})-q_1\right),  
      \\
G_{\rm KFS}^{\frac{3}{2}+2}
=   {}&
  \left(q_2-q_1\right) \left(\left(\theta_2^0-\theta_1^0\right) p_1 p_2
   q_1-\theta^{\infty}_2 p_2^2 q_1+p_1 p_2^2 q_1 q_2+p_1^2 p_2
   q_1^2+p_1 q_1+p_1 p_2 \tilde{t}\right)
   \\
   &
   +\theta_2^0
   \left(\tilde{t} H_{\rm III}(D_7)(-\theta_1^0;\tilde{t};q_1,p_1)-q_1+p_1p_2q_1^2\right)
   -\theta^{\infty} _2 \tilde{t} H_{\rm III}(D_7)(-\theta_1^0;\tilde{t};q_1,p_2)+p_1p_2q_1^2(\theta_2^0-\theta_2^{\infty}),       
\\
H_{\rm KFS}^{\frac{4}{3}+\frac{3}{2}}
= {}&
H_{\rm III}(D_7)(\theta_1^{\infty};\tilde{t};q_1,p_1)
+H_{\rm III}(D_7)(\delta-\theta_1^{\infty};\tilde{t};q_2,p_2)
-\frac{1}{\tilde{t}}p_1q_1p_2q_2-\left(\frac{p_2}{q_1}+p_1+p_2\right),
\\
G_{\rm KFS}^{\frac{4}{3}+\frac{3}{2}}
= {}& 
\frac{1}{q_1}
\left(
p_2 \left(\theta_1^{\infty}+p_1 q_1-p_2 q_2\right) \left(\tilde{t}-p_1
   q_1^2 q_2\right)+\left(p_1-p_2\right) q_2 q_1^2-\tilde{t}
   \right),
\\ 
H_{\rm KFS}^{\frac{3}{2}+\frac{3}{2}}
= {}&
H_{\rm III}(D_7)(\theta^{\infty}_1-\theta^{\infty}_2;\tilde{t};q_1,p_1)
+H_{\rm III}(D_7)(\delta-\theta^{\infty}_1;\tilde{t};q_2,p_2)
-\frac{1}{\tilde{t}}p_1q_1p_2q_2-(p_1p_2+p_1+p_2),
\\
G_{\rm KFS}^{\frac{3}{2}+\frac{3}{2}}
= {}&
\left(
p_1 p_2q_1-p_2^2 q_2+\theta^{\infty} _1 p_2-1\right)\left(-p_1 \left(q_1q_2-\tilde{t}\right)+\theta^{\infty} _2 q_2\right)
-p_2 \left(q_1 q_2-\tilde{t}\right),
\\
H_{\rm KFS}^{\frac{4}{3}+\frac{4}{3}}
= {}&
\frac{1}{\tilde{t}}
\left(p_1^2q_1^2+\delta q_1p_1-q_1-\frac{\tilde{t}}{q_1}\right)
+H_{\rm III}(D_8)(\tilde{t};q_2,p_2)
+\frac{1}{\tilde{t}}\left(-p_1q_1p_2q_2+\frac{q_1q_2}{\tilde{t}}+q_1+q_2\right),
\\   
G_{\rm KFS}^{\frac{4}{3}+\frac{4}{3}}
= {}& 
\frac{1}{\tilde{t} q_1q_2}
\left(  
\left(p_1
   q_1-p_2 q_2\right) \left(p_1 p_2 q_2^2 q_1^2 \tilde{t}+p_1
   q_2 q_1^2 \tilde{t}+q_2^2 q_1^2\right)
   +\tilde{t}^2 \left(p_1 q_1^2-p_2 q_2^2-q_2\right)
   \right).   
\end{align*}
There are also 4 systems that are ramified derived from $D_6$-Sasano system.
\begin{align*}
H_{\rm KSs}^{\frac{3}{2}+2}
= {}& 
H_{\rm III}(D_7)(\theta_0+2\theta^{\infty}_2;\tilde{t};q_1,p_1)
+H_{\rm III}(D_7)(-\theta_0;\tilde{t};q_2,p_2)
+\frac{1}{\tilde{t}}(2p_2q_1(p_1q_1-\theta_0-\theta^{\infty}_1)-q_1),
\\
G_{\rm KSs}^{\frac{3}{2}+2}
= {}&
   p_1^2 p_2^2 q_1^4
   -p_1^2 q_1^3
   -2 p_1^2 p_2^2 q_2 q_1^3
   -2 p_1 p_2^2 \theta _0 q_1^3
   +p_1^2 p_2 \theta _0 q_1^3
   -2 p_1 p_2^2 \theta^{\infty} _1 q_1^3
   +\tilde{t} p_1 p_2^2 q_1^2
   +p_1^2 p_2^2 q_2^2 q_1^2
   \\
   &   
   +p_2^2 \theta _0^2 q_1^2
   -p_1 p_2 \theta _0^2 q_1^2
   +p_2^2 (\theta^{\infty} _1)^2 q_1^2
   +p_1^2 q_2 q_1^2
   +p_1 \theta _0 q_1^2
   +p_1 p_2^2 q_2 \theta _0 q_1^2
   -p_1^2 p_2 q_2 \theta _0 q_1^2
   +p_1 \theta^{\infty} _1 q_1^2
   \\
   &   
   +2 p_1 p_2^2 q_2 \theta^{\infty} _1 q_1^2
   +2 p_2^2 \theta _0 \theta^{\infty} _1 q_1^2
   -p_1 p_2 \theta _0 \theta^{\infty} _1 q_1^2
   -p_1 \theta^{\infty} _2 q_1^2
   -2 p_1 p_2^2 q_2 \theta^{\infty} _2 q_1^2
   +p_1 p_2 \theta _0 \theta^{\infty} _2 q_1^2
   \\
   &   
   +p_2^2 q_2 \theta _0^2 q_1
   -p_1 p_2 q_2 \theta _0^2 q_1
   -\tilde{t} p_1 q_1
   -2 \tilde{t} p_1 p_2^2 q_2 q_1
   -\tilde{t} p_2^2 \theta _0 q_1
   +p_1 p_2^2 q_2^2 \theta _0 q_1
   +\tilde{t} p_1 p_2 \theta _0 q_1
   \\
   &   
   +p_1 q_2 \theta _0 q_1
   -\tilde{t} p_2^2 \theta^{\infty} _1 q_1
   +p_2^2 q_2 \theta _0 \theta^{\infty} _1 q_1
   +2 p_1 p_2^2 q_2^2 \theta^{\infty} _2 q_1
   -p_2 \theta _0^2 \theta^{\infty} _2 q_1
   +2 p_1 q_2 \theta^{\infty} _2 q_1
   \\
   & 
   +2 p_2^2 q_2 \theta _0 \theta^{\infty} _2 q_1
   -2 p_1 p_2 q_2 \theta _0 \theta^{\infty} _2 q_1
   +\theta _0 \theta^{\infty} _2 q_1
   +2 p_2^2 q_2 \theta^{\infty} _1 \theta^{\infty} _2 q_1
   -p_2 \theta _0 \theta^{\infty} _1 \theta^{\infty} _2 q_1
   +\theta^{\infty} _1 \theta^{\infty} _2 q_1
   \\
   &
   +\tilde{t} p_1 p_2^2 q_2^2
   +p_2^2 q_2^2 (\theta^{\infty} _2)^2
   +\tilde{t} p_2 (\theta^{\infty} _2)^2
   +q_2 (\theta^{\infty} _2)^2
   -p_2 q_2 \theta _0 (\theta^{\infty} _2)^2
   +\tilde{t} p_1 q_2
   +\tilde{t} p_2^2 q_2 \theta _0
   -\tilde{t} p_1 p_2 q_2 \theta _0
   \\
   &
   +\tilde{t} p_2^2 q_2 \theta^{\infty} _1
   -p_2 q_2 \theta _0^2 \theta^{\infty} _2
   +p_2^2 q_2^2 \theta _0 \theta^{\infty} _2
   +2 \tilde{t} p_2 \theta _0 \theta^{\infty} _2
   +q_2 \theta _0 \theta^{\infty}
   _2+\tilde{t} p_2 \theta^{\infty} _1 \theta^{\infty} _2,
\\
H_{\rm KSs}^{\frac{4}{3}+2}
= {}&
H_{\rm III}(D_7)(\theta_0+2\theta^{\infty}_2;\tilde{t};q_1,p_1)
+H_{\rm III}(D_7)(-\theta_0;\tilde{t};q_2,p_2)
-\frac{1}{\tilde{t}}\left(2p_2q_1+q_1+\tilde{t}p_2\right),
\\
G_{\rm KSs}^{\frac{4}{3}+2}
= {}&
   -3 \theta _0^2 \theta^{\infty} _1 p_2 q_2
   -\theta _0 (\theta^{\infty} _1)^2 p_2 q_2
   +3 \theta _0 \theta^{\infty} _1 p_2^2 q_2^2
   +\theta _0 \theta^{\infty} _1 p_2 q_1
   +2 \theta _0 \theta^{\infty} _1 p_1 p_2 q_1 q_2
   -2 \theta _0^3 p_2 q_2
   \\
   &
   +2 \theta _0^2 p_2^2 q_2^2
   +2 \theta _0^2 p_2 q_1
   +3 \theta _0^2 p_1 p_2 q_1 q_2
   -\theta _0 p_1 p_2 q_1^2
   -3 \theta _0 p_1 p_2^2 q_1 q_2^2
   -\theta _0 p_1^2 p_2 q_1^2 q_2
   \\
   &
   -3 \theta _0 p_2^2 q_1 q_2
   -3 \theta _0 p_1 q_1 q_2
   +(\theta^{\infty} _1)^2 p_2^2 q_2^2
   -2 \theta^{\infty} _1 p_1 p_2^2 q_1 q_2^2
   -2 \theta^{\infty} _1 p_2^2 q_1 q_2
   -2 \theta^{\infty} _1 p_1 q_1 q_2
   \\
   &   
   -\theta _0 p_1 p_2 q_2 \tilde{t}
   +p_1 p_2^2 q_2^2 \tilde{t}
   +p_2^2 q_2 \tilde{t}
   +p_1 q_2 \tilde{t}
   +p_2^2 q_1^2
   +p_1 q_1^2+p_1^2 p_2^2 q_1^2 q_2^2
   +p_1^2 q_1^2 q_2
   \\
   &   
   +2 p_1 p_2^2 q_1^2 q_2
   -2 \theta _0 p_2 \tilde{t}
   -\theta^{\infty} _1 p_2 \tilde{t}
   +3 \theta _0 \theta^{\infty} _1 q_2
   +2 \theta _0^2 q_2
   -2 \theta _0 q_1
   +(\theta^{\infty} _1)^2 q_2
   -\theta^{\infty} _1 q_1,
\\
H_{\rm KSs}^{\frac{5}{4}+2}
= {}&
H_{\rm III}(D_8)(\tilde{t};q_1,p_1)+H_{\rm III}(D_7)(-\theta_0,\tilde{t};q_2,p_2)
+2\frac{p_2}{q_1}-\frac{(\theta_0+1-\delta)p_1q_1}{\tilde{t}}+\frac{1}{q_1}-p_2,
\\
G_{\rm KSs}^{\frac{5}{4}+2}
= {}&
   \frac{1}{q_1^2}
   \biggl(
   p_2^2 
   \left(
   p_1 q_2 q_1^2 \left(2
   \tilde{t}-\theta _0 q_1 q_2\right)+p_1^2 q_2^2
   q_1^4-\theta _0 q_2 q_1 \tilde{t}-q_2^2
   q_1^3+\tilde{t}^2
   \right)
   +q_1^3q_2(p_1^2q_1-1)
   \\
   &     
   -p_2 q_1^2 
   \left(
   \theta _0 p_1^2 q_2 q_1^2
   +\theta _0 p_1 
   \left( \tilde{t}-\theta _0 q_1 q_2 \right)
   -\theta _0 q_2 q_1+\tilde{t}
   \right)
   +q_1^2 p_1 
   \left( \tilde{t}-\theta_0 q_1 q_2 \right)
    \biggr),
     \\
H_{\rm KSs}^{\frac{3}{2}+\frac{5}{4}}
= {}&
\frac{1}{\tilde{t}}\left(p_1^2q_1^2+\delta q_1 p_1-q_1-\tilde{t}/q_1 \right)
+H_{\rm III}(D_8)(\tilde{t};q_2,p_2)
-2\frac{q_1q_2}{\tilde{t}^2}+\frac{q_1+q_2}{\tilde{t}}
,
\\   
G_{\rm KSs}^{\frac{3}{2}+\frac{5}{4}}
= {}&
 \frac{q_1^2q_2^2}{\tilde{t}^2}
 -\frac{p_1^2 q_1^2 \tilde{t}}{q_2}
 +q_1+q_2
+\frac{\tilde{t}^2}{q_1 q_2}
-\frac{1}{4q_1 \tilde{t}}\left(2 p_2q_2+1\right) \left(\tilde{t} \left(2 p_2 q_2+1\right)
   \left(\tilde{t}-p_1^2 q_1^3\right)+4 p_1 q_2 q_1^3\right).
\end{align*}
We have three ramified systems of matrix Painlev\'e equations. 
$G_{{\rm III}(D_7)}^{\rm Mat}$ and $G_{{\rm III}(D_8)}^{\rm Mat}$ are too long and $H_{\rm I}^{\rm Mat}$ is already written in the main part of this paper.
So we skip writing first integrals of autonomous matrix Painlev\'e equations.

\subsection{ Local data of linear equations}\label{ap:linear}
For the classification of linear equations, we need to discern the types of linear equations.
In this subsection, we explain the notation to express spectral types for ramified equations, which is  introduced by Kawakami~\cite{2016arXiv160803927K}. 
The Fuchsian case is explained in \ref{subsec:imd}.

Now we explain the way to obtain the formal canonical form around each non-Fuchsian singular point.
We also explain that these canonical forms can be expressed by the refining sequences of partition. 
Let us assume that the coefficient matrix $A(x)$ of the equation has a singularity at the origin, and that $A(x)$  is expanded in the Laurent series as follows: 
\begin{equation}\label{ex}
\frac{dY}{dx}=
\left(\frac{A^0}{x^{r+1}}+\frac{A^1}{x^r}+\cdots \right)Y.
\end{equation}
Here, $A^j \ (j=0,1,\ldots)$ are $m\times m$ matrices. 
We assume that $A^0$ is diagonalizable.
With an appropriate choice of the gauge matrix, we can assume that $A^{0}$ is diagonal and that its eigenvalues are $t^0_1,\ldots,t^0_m$. 
When $r=0$, then the origin is a regular singular point.
Let us assume that $r>0$.
If $t^0_i\neq t^0_j \quad (1\le i \le l,\ l+1\le j \le m)$, then  a gauge transformation by a formal power series  $Y=P(x)Z$ \ $(P(x)=I+P_1 x+P_2 x^2+\cdots)$ leads to the following form:
\[
 \frac{dZ}{dx}=
 \left(\frac{B^0}{x^{r+1}}+\frac{B^1}{x^r}+\cdots \right)Z.
\]
Here,
 we can transform $B^{i}$ into the following form:
\[
B^i=
\begin{pmatrix}
B^i_{11} & O \\
O & B^i_{22}
\end{pmatrix},\ 
B^i_{11}\in M_l(\mathbb{C}),\ 
B^i_{22}\in M_{m-l}(\mathbb{C}).
\]
With  successive application of this process, the equation (\ref{ex}) is formally decomposed to direct sum of equations whose leading terms have only one eigenvalues  respectively.
When the leading term of the block is diagonalizable, that is when it is scalar matrix, then this part can be canceled by a gauge transformation by a scalar function, so that the equation is reduced to a equation with smaller $r$.

\begin{rem}\label{rem:ram}
 When $A^{0}$ is not diagonalizable,
 in order to decompose the system into equations of smaller sizes,
 we need to take an appropriate covering $x=\xi^k$.
 In that case, the transformation matrix $P(x)$ is a Puiseux series in $x$.
 The equations  with this property are called ramified.
 When we do not need to take coverings, that is when $k=1$, we say that the equations are unramified.
 \qed
\end{rem}
\subsubsection{Unramified non-Fuchsian case}
When the equation (\ref{ex}) is unramified, it can be transformed into the following form:
\[
\frac{dY}{dx}=
\left(\frac{T_0}{x^{r+1}}+\frac{T_1}{x^r}+\cdots+\frac{T_r}{x}+\cdots 
\right)Y.
\]
We can assume that $T_j$'s are diagonal matrices and that $T_0=A_0$.
Furthermore, we can eliminate the 
regular terms
 by an appropriate diagonal matrix with formal power series components.
Thus, the equation (\ref{ex}) can be transformed into the following form by gauge transformation of a formal power series:
\begin{equation}
 \frac{dY}{dx}=
 \left(\frac{T_0}{x^{r+1}}+\frac{T_1}{x^r}+\cdots+\frac{T_r}{x} \right)Y.
 \label{eqn:loc_norm_form}
\end{equation}
If we write the diagonal components of $T_{i}$ as $t^i_j\ (j=1,\ldots,m)$, then a canonical form around the origin can be described by the following data:
\[
\begin{array}{c}
x=0 \\
\overbrace{\begin{array}{cccc}
   t^0_1 & t^1_1 & \ldots & t^r_1 \\
   \vdots & \vdots & & \vdots \\
   t^0_m & t^1_m & \ldots & t^r_m
           \end{array}}
\end{array}_{.}
\]
We write down such formal canonical forms for each singular point, and put them together.
This kind of table is called the Riemann scheme of the linear equation. 
As we can see from the procedure to obtain the canonical form,  the leftmost column splits into several groups as equivalence class of values.
In the second column from the left, these groups splits further, and so on, we get a nested columns.

We describe such nesting structure by refining sequences of partitions of $m$ and call it the spectral type of the singular point.
We line up such spectral types of each singular point, and separate them by commas. 
We call it the spectral type of the equation. 
In such a case,
\[
\exp\left(-\frac{T_{0}}{rx^r}+\dots-\frac{T_{r-1}}{x}\right)x^{T_{r}}
\]
is the fundamental solution matrix for the formal canonical form (\ref{eqn:loc_norm_form}).
The degree $r$ of the polynomial is called the Poincar\'e rank of the singular point.
When the singular point is of regular type, the Poincar\'e rank is $0$.
When the equation is ramified, this part is a polynomial in $x^{-1/k}$, and the Poincar\'e rank is non-integer rational number.
If the thing we want to express is just the Poincar\'e ranks at each singularity, we attach Poincar\'e rank plus 1 to each singularity, and line them up, and separate them by $+$ signs. 
When the equation is unramified, Poincar\'e rank plus 1 is as same as numbers of columns appeared in refining sequences of partitions at each singularities.
\begin{eg}
For instance, let us consider the following normal form:
\vspace{-1mm}
\begin{equation*}
\frac{d\hat{Y}(x)}{dx}=\{\frac{1}{x^{3}}\begin{pmatrix}
a  &  0  &  0  &0\\
0  &  a  &  0  &0\\
0  &  0  &  b  &0\\
0  &  0  &  0  &b
\end{pmatrix}
+\frac{1}{x^2}\begin{pmatrix}
c  &  0  &  0  &0\\
0  &  c  &  0  &0\\
0  &  0  &  d  &0\\
0  &  0  &  0  &  e
\end{pmatrix}
+\frac{1}{x}\begin{pmatrix}
f   &  0  &  0  &0\\
0  &  g  &  0  &0\\
0  &  0  &  h  &0\\
0  &  0  &  0  &i
\end{pmatrix}
\}\hat{Y}.
\end{equation*}
\vspace{-1mm}
We align the diagonal entries as
$  \left\{ \begin{matrix}
     a & a & b & b\\
     c & c & d & e \\
     f & g & h & i
    \end{matrix}\right. .
$
We express the degeneracy of the eigenvalues by
$\begin{Bmatrix}
22\\
211\\
1111
\end{Bmatrix}.
$
Each row expresses a partition of the matrix size $m$.
The partitions in the lower rows are a refinement of a partition in the upper rows.
\end{eg}
In order to express the degeneracy of the eigenvalues briefly, we use parentheses.
Firstly, write the finest partition of $m$ in the lowest row, which expresses the degeneracy of the eigenvalue of $T^{(i)}$.
Secondly, put the numbers that are grouped together in the second lowest partition in  parentheses.
We continue this process until the highest row.

\begin{eg}
The local data of the example above can be expressed concisely using the parentheses.
\begin{align*}
\begin{array}{c}
 x=0  \\
\overbrace{\begin{array}{ccc}
   a & c & f   \\
   a & c & g \\
   b & d & h \\
   b & e & i 
           \end{array}}
\end{array}
=
 \left\{ \begin{matrix}
     a & a & b & b\\
     c & c & d & e \\
     f & g & h & i
    \end{matrix}\right.
\longrightarrow    
\begin{Bmatrix}
22\\
211\\
1111
\end{Bmatrix}
\longrightarrow
((11))((1)(1)),
\hspace{3mm}
\begin{tabular}{|c|c|}
\hline
level 2 & ((11))((1)(1))
\\
level 1 & (11)(1)(1)
\\
level 0 & 1111
\\
\hline
\end{tabular}_{.}
\end{align*}

The way to restore the degeneracy of eigenvalues from the symbol $((11))((1)(1))$ is as follows.
\begin{itemize}
\item Add the numbers in the outermost parenthesis $\rightarrow$ 22
\item Add the numbers in the inner parenthesis $\rightarrow$211
\item Write the numbers in the innermost parentheses $\rightarrow$1111
\end{itemize}
\end{eg}
We express the types of linear equations by aligning such data for each singular point.
\subsubsection{Ramified  case}
We have the following formal normal
form at each singularities~(Hukuhara~\cite{zbMATH03026099}, Levelt~\cite{zbMATH03477550} and Turrittin~\cite{zbMATH03109720}).
Let us assume that $x=0$ is  an irregular singular point.
Then, there exist a positive integer $q$, rational numbers with the common denominator $q$ such that $r_0<r_1<\dots<r_{n-1}<r_{n}=-1$, diagonal matrices $T_0,\dots,T_{n}$, a transformation $z=F(x^{1/q})$ in the class of formal series in $x^{1/q}$, such that the transformed system have the following form;
\begin{align}
\frac{dz}{dx}=(T_0x^{r_0}+\dots+T_{n-1}x^{r_{n-1}}+T_{n}x^{-1})z.
\end{align}
Let us assume that the diagonal matrix $T_{k}$ has $t_i^{k}$ for $i=1,\dots, m$ as diagonal components.
We express the local data by the following table:
\[
\begin{array}{c}
\hspace{5mm} x=0 \hspace{2mm} \left(\frac{n}{q}\right) \\
\overbrace{\begin{array}{cccc}
   t^0_1 & t^1_1 & \ldots & t^n_1 \\
   \vdots & \vdots & & \vdots \\
   t^0_m & t^1_m & \ldots & t^n_m
           \end{array}}
\end{array}_{.}
\]
Let us introduce a way to express such local data compactly with examples.\footnote{Kawakami~\cite{2016arXiv160803927K} devised such notation.}
Let us assume the following normal form;
\[
\frac{dz}{dx}
=
\left\{
c_0
\Omega
x^{-\frac{8}{3}}
+
c_1
\Omega^2
x^{-\frac{7}{3}}
+
c_2I_3x^{-2}
+
c_3
\Omega
x^{-\frac{5}{3}}
+
c_4
\Omega^2
x^{-\frac{4}{3}}
+c_5I_3x^{-1}
\right\}z,
\]
where $\Omega=\operatorname{diag}(1,\omega,\omega^2)$ and $\omega$ is a primitive third root of unity.
This normal form can be expressed as 
\[
\begin{array}{c}
\hspace{5mm} x=0 \hspace{2mm} \left(\frac{5}{3}\right) \\
\overbrace{\begin{array}{cccccc}
   c_0 & c_1 & c_2 & c_3& c_4  & c_5\\
   c_0\omega & c_1\omega & c_2 & c_3\omega& c_4\omega^2 & c_5\\
   c_0\omega^2 & c_1\omega & c_2 & c_3\omega^2& c_4\omega & c_5
           \end{array}}
\end{array}.
\]
Two systems in the lower rows 
$\frac{dz_2}{dx}=\left(c_0\omega x^{-\frac{8}{3}}+\dots\right)z_2$ and
$\frac{dz_3}{dx}=\left(c_0\omega x^{-\frac{8}{3}}+\dots\right)z_3$
can be obtained by the first row 
$\frac{dz_1}{dx}=\left(c_0 x^{-\frac{8}{3}}+\dots\right)z_1$
upon replacement 
$x^{\frac{1}{3}}\mapsto \omega x^{\frac{1}{3}}\mapsto \omega^2 x^{\frac{1}{3}}$. 
Since we have 3 copies of the first equation, we express the local data as $(((((1)))))_{3}$.
\begin{eg}
We show  some other examples.
\begin{align*}
\begin{array}{c}
\hspace{5mm} x=0 \hspace{2mm} \left(\frac{1}{2}\right) \\
\overbrace{\begin{array}{cc}
   \alpha &  \beta \\
   -\alpha & \beta \\
   0 & \gamma \\
   0 & \epsilon
           \end{array}}
\end{array}
\hspace{-3mm}
\to(1)_{2}11,
\begin{array}{c}
\hspace{5mm} x=0 \hspace{2mm} \left(\frac{1}{2}\right) \\
\overbrace{\begin{array}{cc}
   \alpha &  \beta \\
   \alpha\omega & \beta \\
   \alpha \omega^2 & \beta \\
   0 & \gamma
           \end{array}}
\end{array}
\hspace{-3mm}
\to(1)_{3}1,
\begin{array}{c}
\hspace{5mm} x=0 \hspace{2mm} \left(\frac{1}{2}\right) \\
\overbrace{\begin{array}{cc}
   \alpha &  \beta \\
   \alpha & \beta \\
   -\alpha & \beta \\
   -\alpha & \beta
           \end{array}}
\end{array}
\hspace{-3mm}
\to(2)_{2},
\begin{array}{c}
\hspace{5mm} x=0 \hspace{2mm} \left(\frac{1}{2}\right) \\
\overbrace{\begin{array}{cc}
   \alpha &  \beta \\
   \alpha & \gamma \\
   -\alpha  & \beta \\
   -\alpha & \gamma
           \end{array}}
\end{array}
\hspace{-3mm}
\to(11)_{2}.
\end{align*}
\end{eg}

\subsection{The dual graphs of the singular fibers}\label{graph}
We list the dual graphs of singular fibers appeared in the table.
The numbers in circles indicate multiplicities of the components.
We adopt, as in Ogg~\cite{MR0201437}, the following symbol for component $\Gamma$ of singular fibers (Table~\ref{table:comp}).
$K_{X}$ is the canonical divisor of surface $X$.
The matrices next to or below the dual graphs are the monodromy~\cite{MR0369362}.

\begin{table}[ht]
\centering
\begin{tabular}{|c|c|c|c|}
\hline
Symbol & Genus & $\Gamma^2$ & $\Gamma\cdot K_{X}$\\ \hline
A & 1 & -1 & 1\\ \hline
B & 0 & -3 & 1\\ \hline
C & 1 & -2 & 2\\ \hline
D & 0 & -4 & 2\\ \hline
none & 0 & -2 & 0\\
\hline
\end{tabular}
\caption{Components of singular fibers}
\label{table:comp}
\end{table}

\newpage
\begin{multicols}{2}
\noindent	
${\rm I^{*}_{0-0-0}}\colon H_{{\rm Gar}}^{1+1+1+1+1}$
\[
\begin{xy}
(0,10) *+[Fo]{1}="A";
 (-9,4)  *+[Fo]{1}="B";
(0,0) *+[Fo]{2B}="C";
(9,4) *+[Fo]{1}="D" ;
(-9,-4) *+[Fo]{1}="H" ;
(9,-4) *+[Fo]{1}="I"; 
(0,-10) *+[Fo]{1}="J";
\ar@{-} "A";"C"
\ar@{-} "B";"C";
\ar@{-} "C";"D";
\ar@{-} "C";"H";
\ar@{-}"C";"I";
\ar@{-}"C";"J";
\end{xy}
\hspace{4mm}
\begin{matrix}
\begin{pmatrix}
-1 & 0 & 0 & 0
\\
0 & -1 & 0 & 0
\\
0 & 0 & -1 & 0
\\
0 & 0 & 0 & -1
\end{pmatrix}
\end{matrix}
\]

\columnbreak

${\rm I^{*}_{1-0-0}}\colon H_{{\rm Gar}}^{2+1+1+1}$
\[
\begin{xy}
(7,-7) *+[Fo]{1}="I";
(-7,-7)  *+[Fo]{1}="K";
(-7,7)  *+[Fo]{1}="L";
 (0,0)  *+[Fo]{2B}="A";
(10,0)  *+[Fo]{2}="C";
(7,7) *+[Fo]{1}="D";
(17,7) *+[Fo]{1}="E";
(17,-7) *+[Fo]{1}="G";
\ar@{-} "A";"C";
\ar@{-} "A";"D";
\ar@{-} "C";"E";
\ar@{-} "C";"G";
\ar@{-} "I";"A";
\ar@{-} "K";"A";
\ar@{-} "L";"A";
\end{xy}
\hspace{4mm}
\begin{pmatrix}
-1 & 0 & 0 & 0
\\
0 & -1 & 0 & -1
\\
0 & 0 & -1 & 0 
\\
0 & 0 & 0 & -1 
\end{pmatrix}
\]
\end{multicols}

\begin{multicols}{2}
\noindent	
${\rm I^{*}_{2-0-0}}\colon H_{{\rm Gar}}^{\frac{3}{2}+1+1+1}$
\[
\begin{xy}
(7,-7) *+[Fo]{1}="I";
(-7,-7)  *+[Fo]{1}="K";
(-7,7)  *+[Fo]{1}="L";
 (0,0)  *+[Fo]{2B}="A";
(10,0) *+[Fo]{2}="B";
(20,0)  *+[Fo]{2}="C";
(7,7) *+[Fo]{1}="D";
(27,7) *+[Fo]{1}="E";
(27,-7) *+[Fo]{1}="G";
\ar@{-} "A";"B";
\ar@{-} "B";"C";
\ar@{-} "A";"D";
\ar@{-} "C";"E";
\ar@{-} "C";"G";
\ar@{-} "I";"A";
\ar@{-} "K";"A";
\ar@{-} "L";"A";
\end{xy}
\hspace{4mm}
\begin{pmatrix}
-1 & 0 & 0 & 0
\\
0 & -1 & 0 & -2
\\
0 & 0 & -1 & 0 
\\
0 & 0 & 0 & -1 
\end{pmatrix}
\]

\columnbreak

\noindent	
${\rm I^{*}_{1-1-0}}\colon H_{{\rm Gar}}^{2+2+1}$
\[
\begin{xy}
(-17,7) *+[Fo]{1}="I";
(-17,-7) *+[Fo]{1}="J";
(-10,0)  *+[Fo]{2}="K";
(0,10)  *+[Fo]{1}="L";
 (0,0)  *+[Fo]{2B}="A";
(10,0)  *+[Fo]{2}="C";
(0,-10) *+[Fo]{1}="F";
(17,7) *+[Fo]{1}="E";
(17,-7) *+[Fo]{1}="D";
\ar@{-} "A";"C";
\ar@{-} "C";"D";
\ar@{-} "C";"E";
\ar@{-} "A";"F";
\ar@{-} "A";"K";
\ar@{-} "I";"K";
\ar@{-} "J";"K";
\ar@{-} "L";"A";
\end{xy}
\hspace{4mm}
\begin{pmatrix}
-1 & 0 & -1 & 0
\\
0 & -1 & 0 & -1
\\
0 & 0 & -1 & 0 
\\
0 & 0 & 0 & -1 
\end{pmatrix}
\]

\columnbreak

\end{multicols}

\begin{multicols}{2}
\noindent	
${\rm I^{*}_{1-2-0}}\colon H_{{\rm Gar}}^{\frac{3}{2}+2+1}$
\[
\begin{xy}
(-17,7) *+[Fo]{1}="I";
(-17,-7) *+[Fo]{1}="J";
(-10,0)  *+[Fo]{2}="K";
(0,10)  *+[Fo]{1}="L";
 (0,0)  *+[Fo]{2B}="A";
(10,0) *+[Fo]{2}="B";
(20,0)  *+[Fo]{2}="C";
(0,-10) *+[Fo]{1}="F";
(27,7) *+[Fo]{1}="E";
(27,-7) *+[Fo]{1}="D";
\ar@{-} "A";"B";
\ar@{-} "B";"C";
\ar@{-} "C";"D";
\ar@{-} "C";"E";
\ar@{-} "A";"F";
\ar@{-} "A";"K";
\ar@{-} "I";"K";
\ar@{-} "J";"K";
\ar@{-} "L";"A";
\end{xy}
\hspace{4mm}
\begin{pmatrix}
-1 & 0 & -2 & 0
\\
0 & -1 & 0 & -1
\\
0 & 0 & -1 & 0 
\\
0 & 0 & 0 & -1 
\end{pmatrix}
\]

\columnbreak

\end{multicols}

\begin{multicols}{2}
\noindent	
${\rm I^{*}_{2-2-0}}\colon H_{{\rm Gar}}^{\frac{3}{2}+\frac{3}{2}+1}$
\[
\begin{xy}
(-27,7) *+[Fo]{1}="I";
(-27,-7) *+[Fo]{1}="J";
(-20,0)  *+[Fo]{2}="K";
(0,10)  *+[Fo]{1}="L";
(-10,0) *+[Fo]{2}="G";
 (0,0)  *+[Fo]{2B}="A";
(10,0) *+[Fo]{2}="B";
(20,0)  *+[Fo]{2}="C";
(0,-10) *+[Fo]{1}="F";
(27,7) *+[Fo]{1}="E";
(27,-7) *+[Fo]{1}="D";
\ar@{-} "A";"B";
\ar@{-} "B";"C";
\ar@{-} "C";"D";
\ar@{-} "C";"E";
\ar@{-} "A";"F";
\ar@{-} "A";"G";
\ar@{-} "I";"K";
\ar@{-} "J";"K";
\ar@{-} "K";"G";
\ar@{-} "L";"A";
\end{xy}
\hspace{4mm}
\begin{pmatrix}
-1 & 0 & -2 & 0
\\
0 & -1 & 0 & -2 
\\
0 & 0 & -1 & 0 
\\
0 & 0 & 0 & -1 
\end{pmatrix}
\]
\columnbreak

\end{multicols}

\begin{multicols}{2}
\noindent	
${\rm I^{*}_{0}-IV^{*}-(-1)}\colon H_{{\rm Gar}}^{3+1+1}$
\[
\begin{xy}
(0,-10) *+[Fo]{1}="I";
(-10,0)  *+[Fo]{1}="K";
(0,10)  *+[Fo]{1}="L";
 (0,0)  *+[Fo]{2B}="A";
(10,0) *+[Fo]{3}="B";
(20,0)  *+[Fo]{2}="C";
(10,10) *+[Fo]{2}="D";
(10,20) *+[Fo]{1}="F";
(30,0) *+[Fo]{1}="E";
\ar@{-} "A";"B";
\ar@{-} "B";"C";
\ar@{-} "B";"D";
\ar@{-} "C";"E";
\ar@{-} "D";"F";
\ar@{-} "I";"A";
\ar@{-} "K";"A";
\ar@{-} "L";"A";
\end{xy}
\hspace{4mm}
\begin{pmatrix}
-1 & 0 & -1 & 0
\\
0 & -1 & 0 & 0 
\\
1 & 0 & 0 & 0 
\\
0 & 0 & 0 & -1 
\end{pmatrix}
\]

\columnbreak

\end{multicols}

\begin{multicols}{2}
\noindent	
${\rm I^{*}_{0}-III^{*}-(-1)}\colon H_{{\rm Gar}}^{\frac{5}{2}+1+1}$
\[
\begin{xy}
(0,-10) *+[Fo]{1}="I";
(-10,0)  *+[Fo]{1}="K";
(0,10)  *+[Fo]{1}="L";
 (0,0)  *+[Fo]{2B}="A";
(10,0) *+[Fo]{3}="B";
(20,0)  *+[Fo]{4}="C";
(20,10) *+[Fo]{2}="D";
(30,0) *+[Fo]{3}="E";
(40,0) *+[Fo]{2}="G";
(50,0) *+[Fo]{1}="H";
\ar@{-} "A";"B";
\ar@{-} "B";"C";
\ar@{-} "C";"D";
\ar@{-} "C";"E";
\ar@{-} "E";"G";
\ar@{-} "I";"A";
\ar@{-} "G";"H";
\ar@{-} "K";"A";
\ar@{-} "L";"A";
\end{xy}
\hspace{4mm}
\begin{pmatrix}
0 & 0 & -1 & 0
\\
0 & -1 & 0 & 0 
\\
1 & 0 & 0 & 0 
\\
0 & 0 & 0 & -1 
\end{pmatrix}
\]

\columnbreak

\end{multicols}

\begin{multicols}{2}
\noindent	
${\rm III^*-II^{*}_{0}-(-1)}\colon H_{{\rm Gar}}^{4+1}$
\[
\begin{xy}
(-10,0)  *+[Fo]{1}="K";
 (0,0)  *+[Fo]{2}="A";
(10,0) *+[Fo]{3}="B";
(20,0)  *+[Fo]{4}="C";
(20,10) *+[Fo]{2B}="D";
(13,16)  *+[Fo]{1}="L";
(27,16) *+[Fo]{1}="J";
(30,0) *+[Fo]{3}="E";
(40,0) *+[Fo]{2}="F";
(50,0) *+[Fo]{1}="G";
\ar@{-} "A";"B";
\ar@{-} "B";"C";
\ar@{-} "C";"D";
\ar@{-} "C";"E";
\ar@{-} "E";"F";
\ar@{-} "F";"G";
\ar@{-} "D";"L";
\ar@{-} "D";"J";
\ar@{-} "K";"A";
\end{xy}
\hspace{4mm}
\begin{pmatrix}
0 & 0 & -1 & 0
\\
0 & -1 & -1 & -1
\\
1 & 0 & 0 & 1 
\\
0 & 0 & 0 & -1 
\end{pmatrix}
\]

\columnbreak

\end{multicols}

\begin{multicols}{2}
\noindent	
${\rm II^*-II^{*}_{0}-(-1)}\colon H_{{\rm Gar}}^{\frac{7}{2}+1}$
\[
\begin{xy}
(-27,-7)  *+[Fo]{1}="L";
(-27,7) *+[Fo]{1}="J";
(-20,0) *+[Fo]{2B}="I";
(-10,0)  *+[Fo]{4}="K";
 (0,0)  *+[Fo]{6}="A";
(10,0) *+[Fo]{5}="B";
(20,0)  *+[Fo]{4}="C";
(0,10) *+[Fo]{3}="D";
(30,0) *+[Fo]{3}="E";
(40,0) *+[Fo]{2}="F";
(50,0) *+[Fo]{1}="G";
\ar@{-} "A";"B";
\ar@{-} "B";"C";
\ar@{-} "A";"D";
\ar@{-} "C";"E";
\ar@{-} "E";"F";
\ar@{-} "F";"G";
\ar@{-} "I";"L";
\ar@{-} "I";"J";
\ar@{-} "K";"I";
\ar@{-} "K";"A";
\end{xy}
\hspace{4mm}
\begin{pmatrix}
0 & 0 & -1 & 0
\\
0 & -1 & 1 & 0
\\
1 & 0 & 1 & -1 
\\
0 & 0 & 0 & -1 
\end{pmatrix}
\]

\columnbreak

\end{multicols}

\begin{multicols}{2}
\noindent	
${\rm IV^{*}-I^{*}_{1}-(-1)}\colon H_{{\rm Gar}}^{3+2}$
\[
\begin{xy}
(-18,5) *+[Fo]{1}="I";
(-18,-5) *+[Fo]{1}="J";
(-10,0)  *+[Fo]{2}="K";
(0,10)  *+[Fo]{1}="L";
 (0,0)  *+[Fo]{2B}="A";
(10,0) *+[Fo]{3}="B";
(20,0)  *+[Fo]{2}="C";
(10,10) *+[Fo]{2}="D";
(10,20) *+[Fo]{1}="F";
(30,0) *+[Fo]{1}="E";
\ar@{-} "A";"B";
\ar@{-} "B";"C";
\ar@{-} "B";"D";
\ar@{-} "C";"E";
\ar@{-} "D";"F";
\ar@{-} "I";"K";
\ar@{-} "J";"K";
\ar@{-} "K";"A";
\ar@{-} "L";"A";
\end{xy}
\hspace{4mm}
\begin{pmatrix}
-1 & 0 & -1 & 0
\\
0 & -1 & 0 & -1 
\\
1 & 0 & 0 & 0 
\\
0 & 0 & 0 & -1 
\end{pmatrix}
\]

\columnbreak

\end{multicols}

\begin{multicols}{2}
\noindent	
${\rm IV^{*}-I^{*}_{2}-(-1)}\colon H_{{\rm Gar}}^{3+\frac{3}{2}}$
\[
\begin{xy}
(-28,5) *+[Fo]{1}="I";
(-28,-5) *+[Fo]{1}="J";
(-20,0)  *+[Fo]{2}="K";
(0,10)  *+[Fo]{1}="L";
(-10,0) *+[Fo]{2}="G";
 (0,0)  *+[Fo]{2B}="A";
(10,0) *+[Fo]{3}="B";
(20,0)  *+[Fo]{2}="C";
(10,10) *+[Fo]{2}="D";
(10,20) *+[Fo]{1}="F";
(30,0) *+[Fo]{1}="E";
\ar@{-} "A";"B";
\ar@{-} "B";"C";
\ar@{-} "B";"D";
\ar@{-} "C";"E";
\ar@{-} "D";"F";
\ar@{-} "A";"G";
\ar@{-} "I";"K";
\ar@{-} "J";"K";
\ar@{-} "K";"G";
\ar@{-} "L";"A";
\end{xy}
\hspace{4mm}
\begin{pmatrix}
-1 & 0 & -1 & 0
\\
0 & -1 & 0 & -2 
\\
1 & 0 & 0 & 0 
\\
0 & 0 & 0 & -1 
\end{pmatrix}
\]

\columnbreak

\end{multicols}

\begin{multicols}{2}
\noindent	
${\rm III^*-I^{*}_{1}-(-1)}\colon H_{{\rm Gar}}^{\frac{5}{2}+2}$
\[
\begin{xy}
(-18,5) *+[Fo]{1}="I";
(-18,-5) *+[Fo]{1}="J";
(-10,0)  *+[Fo]{2}="K";
 (0,0)  *+[Fo]{2B}="A";
(10,0) *+[Fo]{3}="B";
(20,0)  *+[Fo]{4}="C";
(20,10) *+[Fo]{2}="D";
(30,0) *+[Fo]{3}="E";
(40,0) *+[Fo]{2}="F";
(50,0) *+[Fo]{1}="G";
\ar@{-} "A";"B";
\ar@{-} "B";"C";
\ar@{-} "C";"D";
\ar@{-} "C";"E";
\ar@{-} "E";"F";
\ar@{-} "F";"G";
\ar@{-} "I";"K";
\ar@{-} "J";"K";
\ar@{-} "K";"A";
\end{xy}
\hspace{4mm}
\begin{pmatrix}
0 & 0 & -1 & 0
\\
0 & -1 & 0 & -1
\\
1 & 0 & 0 & 0 
\\
0 & 0 & 0 & -1 
\end{pmatrix}
\]

\columnbreak

\end{multicols}

\begin{multicols}{2}
\noindent	
${\rm III^*-I^{*}_{2}-(-1)}\colon H_{{\rm Gar}}^{\frac{5}{2}+\frac{3}{2}}$
\[
\begin{xy}
(-28,5) *+[Fo]{1}="I";
(-28,-5) *+[Fo]{1}="J";
(-20,0)  *+[Fo]{2}="K";
(-10,0)  *+[Fo]{2}="L";
 (0,0)  *+[Fo]{2B}="A";
(10,0) *+[Fo]{3}="B";
(20,0)  *+[Fo]{4}="C";
(20,10) *+[Fo]{2}="D";
(30,0) *+[Fo]{3}="E";
(40,0) *+[Fo]{2}="F";
(50,0) *+[Fo]{1}="G";
\ar@{-} "A";"B";
\ar@{-} "B";"C";
\ar@{-} "C";"D";
\ar@{-} "C";"E";
\ar@{-} "E";"F";
\ar@{-} "F";"G";
\ar@{-} "I";"K";
\ar@{-} "J";"K";
\ar@{-} "K";"L";
\ar@{-} "L";"A";
\end{xy}
\hspace{4mm}
\begin{pmatrix}
0 & 0 & -1 & 0
\\
0 & -1 & 0 & -2 
\\
1 & 0 & 0 & 0 
\\
0 & 0 & 0 & -1 
\end{pmatrix}
\]

\columnbreak

\end{multicols}

\begin{multicols}{2}
\noindent	
${\rm II^{*}-II^{*}_{0}}\colon H_{{\rm Gar}}^{\frac{7}{2}+1}$
\[
\begin{xy}
(-28,5) *+[Fo]{1}="I";
(-28,-5) *+[Fo]{1}="J";
(-20,0)  *+[Fo]{2B}="K";
(-10,0)  *+[Fo]{4}="L";
 (0,0)  *+[Fo]{6}="A";
 (0,10) *+[Fo]{3}="D";
(10,0) *+[Fo]{5}="B";
(20,0)  *+[Fo]{4}="C";
(30,0) *+[Fo]{3}="E";
(40,0) *+[Fo]{2}="F";
(50,0) *+[Fo]{1}="G";
\ar@{-} "A";"B";
\ar@{-} "B";"C";
\ar@{-} "A";"D";
\ar@{-} "C";"E";
\ar@{-} "E";"F";
\ar@{-} "F";"G";
\ar@{-} "I";"K";
\ar@{-} "J";"K";
\ar@{-} "K";"L";
\ar@{-} "L";"A";
\end{xy}
\hspace{4mm}
\begin{pmatrix}
0 & 0 & -1 & 0
\\
0 & -1 & 1 & 0 
\\
1 & 0 & 1 & -1 
\\
0 & 0 & 0 & -1 
\end{pmatrix}
\]

\columnbreak

\end{multicols}

\begin{multicols}{2}
\noindent	
${\rm IX-3}\colon H_{{\rm Gar}}^{5}$
\[
\begin{xy}
(-40,0) *+[Fo]{1}="J";
(-30,0) *+[Fo]{2}="I";
(-20,0)  *+[Fo]{3}="K";
(-10,0)  *+[Fo]{4}="L";
 (0,0)  *+[Fo]{5}="A";
(10,0) *+[Fo]{4}="B";
(20,0)  *+[Fo]{3}="C";
(0,10) *+[Fo]{2B}="D";
(0,20) *+[Fo]{1}="G";
(30,0) *+[Fo]{2}="E";
(40,0) *+[Fo]{1}="F";
\ar@{-} "A";"B";
\ar@{-} "B";"C";
\ar@{-} "A";"D";
\ar@{-} "C";"E";
\ar@{-} "E";"F";
\ar@{-} "D";"G";
\ar@{-} "I";"K";
\ar@{-} "J";"I";
\ar@{-} "K";"L";
\ar@{-} "L";"A";
\end{xy}
\hspace{4mm}
\begin{pmatrix}
0 & -1 & -1 & 0
\\
-1 & 0 & 0 & -1 
\\
1 & -1 & -1 & 0 
\\
0 & 1 & 0 & 0
\end{pmatrix}
\]

\columnbreak

\end{multicols}

\begin{multicols}{2}
\noindent	
${\rm VII^{*}}\colon H_{{\rm Gar}}^{\frac{9}{2}}$
\[
\begin{xy}
(-50,0) *+[Fo]{1}="M";
(-40,0) *+[Fo]{2B}="J";
(-30,0) *+[Fo]{5}="I";
(-20,0)  *+[Fo]{8}="K";
(-20,10) *+[Fo]{4}="D";
(-10,0)  *+[Fo]{7}="L";
 (0,0)  *+[Fo]{6}="A";
(10,0) *+[Fo]{5}="B";
(20,0)  *+[Fo]{4}="C";
(30,0) *+[Fo]{3}="E";
(40,0) *+[Fo]{2}="F";
(50,0) *+[Fo]{1}="G";
\ar@{-} "A";"B";
\ar@{-} "B";"C";
\ar@{-} "K";"D";
\ar@{-} "C";"E";
\ar@{-} "E";"F";
\ar@{-} "F";"G";
\ar@{-} "I";"K";
\ar@{-} "J";"M";
\ar@{-} "J";"I";
\ar@{-} "K";"L";
\ar@{-} "L";"A";
\end{xy}
\hspace{4mm}
\begin{pmatrix}
0 & -1 & -1 & 0
\\
-1 & 1 & 0 & -1 
\\
1 & -1 & -1 & 0 
\\
0 & 1 & 0 & 0
\end{pmatrix}
\]

\columnbreak

\end{multicols}

\begin{multicols}{2}
\noindent
para[3], ${\rm II_{3-0}}\colon H_{{\rm FS}}^{A_5}$
\[
\begin{xy}
(0,9) *+[Fo]{2}="A";
 (-10,0)  *+[Fo]{B}="B";
(10,0) *+[Fo]{B}="D" ;
(0,-9) *+[Fo]{1}="H" ;
(-7,15) *+[Fo]{1}="J";
(7,15) *+[Fo]{1}="K";
\ar@{-}"A";"K";
\ar@{-}"A";"B";
\ar@{-}"A";"D";
\ar@{-}"D";"H";
\ar@{-}"B";"H";
\ar@{-}"A";"J";
\end{xy}
\hspace{4mm}
\begin{matrix}
\begin{pmatrix}
-1 & 0 & 0 & 0
\\
-1 & 1 & 0 & 3 
\\
0 & 0 & 0 & 1
\\
0 & 0 & 0 & 1 
\end{pmatrix}
\end{matrix}
\]

\columnbreak
\noindent
para[5], ${\rm II_{3-1}}\colon H_{{\rm FS}}^{A_4}$
\[
\begin{xy}
(-9,-2)  *+[Fo]{B}="B";
(9,-2) *+[Fo]{B}="D" ;
(0,-9) *+[Fo]{1}="H" ;
(6,7) *+[Fo]{2}="L";
(-6,7) *+[Fo]{2}="I"; 
(-7,15) *+[Fo]{1}="J";
(7,15) *+[Fo]{1}="K";
\ar@{-}"D";"L";
\ar@{-}"L";"K";
\ar@{-}"I";"B";
\ar@{-}"L";"I";
\ar@{-}"D";"H";
\ar@{-}"B";"H";
\ar@{-}"J";"I";
\end{xy}
\hspace{4mm}
\begin{matrix}
\begin{pmatrix}
-1 & 0 & -1 & 0
\\
1 & 1 & 1 & 3 
\\
0 & 0 & -1 & 1
\\
0 & 0 & 0 & 1 
\end{pmatrix}
\end{matrix}
\]

\columnbreak
\end{multicols}

\begin{multicols}{2}
\noindent
para[5], ${\rm II_{3-2}}\colon H_{{\rm FS}}^{A_3}$
\[
\begin{xy}
(0,10) *+[Fo]{2}="A";
(-8,-2)  *+[Fo]{B}="B";
(8,-2) *+[Fo]{B}="D" ;
(0,-8) *+[Fo]{1}="H" ;
(8,6) *+[Fo]{2}="L";
(-8,6) *+[Fo]{2}="I"; 
(-13,11) *+[Fo]{1}="J";
(13,11) *+[Fo]{1}="K";
\ar@{-}"D";"L";
\ar@{-}"L";"K";
\ar@{-}"I";"B";
\ar@{-}"A";"L";
\ar@{-}"D";"H";
\ar@{-}"B";"H";
\ar@{-}"J";"I";
\ar@{-}"A";"I";
\end{xy}
\begin{matrix}
\begin{pmatrix}
-1 & 0 & -2 & 0
\\
1 & 1 & 2 & 3 
\\
0 & 0 & -1 & 1
\\
0 & 0 & 0 & 1 
\end{pmatrix}
\end{matrix}
\]
\columnbreak

\noindent

para[5], ${\rm II_{3-3}}\colon H_{{\rm Suz}}^{\frac{3}{2}+2}$
\[
\begin{xy}
(-4,10) *+[Fo]{2}="A";
(4,10) *+[Fo]{2}="C";
(-8,-3)  *+[Fo]{B}="B";
(8,-3) *+[Fo]{B}="D" ;
(0,-8) *+[Fo]{1}="H" ;
(9,4) *+[Fo]{2}="L";
(-9,4) *+[Fo]{2}="I"; 
(-14,8) *+[Fo]{1}="J";
(14,8) *+[Fo]{1}="K";
\ar@{-}"D";"L";
\ar@{-}"L";"K";
\ar@{-}"I";"B";
\ar@{-}"C";"L";
\ar@{-}"A";"C";
\ar@{-}"D";"H";
\ar@{-}"B";"H";
\ar@{-}"J";"I";
\ar@{-}"A";"I";
\end{xy}
\hspace{4mm}
\begin{matrix}
\begin{pmatrix}
-1 & 0 & -3 & 0
\\
1 & 1 & 3 & 3 
\\
0 & 0 & -1 & 1
\\
0 & 0 & 0 & 1 
\end{pmatrix}
\end{matrix}
\]

\end{multicols}

\begin{multicols}{2}
\noindent
para[5], ${\rm II_{3-4}}\colon H_{{\rm KFS}}^{\frac{3}{2}+\frac{3}{2}}$
\[
\begin{xy}
(-7,9) *+[Fo]{2}="A";
(0,12) *+[Fo]{2}="E";
(7,9) *+[Fo]{2}="C";
(-7,-5)  *+[Fo]{B}="B";
(7,-5) *+[Fo]{B}="D" ;
(0,-8) *+[Fo]{1}="H" ;
(10,2) *+[Fo]{2}="L";
(-10,2) *+[Fo]{2}="I"; 
(-16,3) *+[Fo]{1}="J";
(16,3) *+[Fo]{1}="K";
\ar@{-}"D";"L";
\ar@{-}"L";"K";
\ar@{-}"I";"B";
\ar@{-}"C";"L";
\ar@{-}"A";"E";
\ar@{-}"C";"E";
\ar@{-}"D";"H";
\ar@{-}"B";"H";
\ar@{-}"J";"I";
\ar@{-}"A";"I";
\end{xy}
\hspace{4mm}
\begin{matrix}
\begin{pmatrix}
-1 & 0 & -4 & 0
\\
1 & 1 & 4 & 3 
\\
0 & 0 & -1 & 1
\\
0 & 0 & 0 & 1 
\end{pmatrix}
\end{matrix}
\]

\columnbreak

\noindent
para[5], ${\rm II_{3-5}}\colon H_{{\rm KFS}}^{\frac{4}{3}+\frac{3}{2}}$
\[
\begin{xy}
(-9,8) *+[Fo]{2}="A";
(-4,12) *+[Fo]{2}="E";
(4,12) *+[Fo]{2}="F";
(9,8) *+[Fo]{2}="C";
(-7,-5)  *+[Fo]{B}="B";
(7,-5) *+[Fo]{B}="D" ;
(0,-8) *+[Fo]{1}="H" ;
(10,1) *+[Fo]{2}="L";
(-10,1) *+[Fo]{2}="I"; 
(-17,1) *+[Fo]{1}="J";
(17,1) *+[Fo]{1}="K";
\ar@{-}"D";"L";
\ar@{-}"L";"K";
\ar@{-}"I";"B";
\ar@{-}"C";"L";
\ar@{-}"A";"E";
\ar@{-}"F";"E";
\ar@{-}"C";"F";
\ar@{-}"D";"H";
\ar@{-}"B";"H";
\ar@{-}"J";"I";
\ar@{-}"A";"I";
\end{xy}
\hspace{4mm}
\begin{matrix}
\begin{pmatrix}
-1 & 0 & -5 & 0
\\
1 & 1 & 5 & 3 
\\
0 & 0 & -1 & 1
\\
0 & 0 & 0 & 1 
\end{pmatrix}
\end{matrix}
\]
\end{multicols}

\begin{multicols}{2}
\noindent
para[5], ${\rm II_{3-6}}\colon H_{{\rm KFS}}^{\frac{4}{3}+\frac{4}{3}}$
\[
\begin{xy}
(-11,6) *+[Fo]{2}="A";
(-7,11) *+[Fo]{2}="E";
(0,13) *+[Fo]{2}="G";
(7,11) *+[Fo]{2}="F";
(11,6) *+[Fo]{2}="C";
(-7,-7)  *+[Fo]{B}="B";
(7,-7) *+[Fo]{B}="D" ;
(0,-9) *+[Fo]{1}="H" ;
(11,-1) *+[Fo]{2}="L";
(-11,-1) *+[Fo]{2}="I"; 
(-18,-3) *+[Fo]{1}="J";
(18,-3) *+[Fo]{1}="K";
\ar@{-}"D";"L";
\ar@{-}"L";"K";
\ar@{-}"I";"B";
\ar@{-}"C";"L";
\ar@{-}"A";"E";
\ar@{-}"E";"G";
\ar@{-}"F";"G";
\ar@{-}"C";"F";
\ar@{-}"D";"H";
\ar@{-}"B";"H";
\ar@{-}"J";"I";
\ar@{-}"A";"I";
\end{xy}
\hspace{4mm}
\begin{matrix}
\begin{pmatrix}
-1 & 0 & -6 & 0
\\
1 & 1 & 6 & 3 
\\
0 & 0 & -1 & 1
\\
0 & 0 & 0 & 1 
\end{pmatrix}
\end{matrix}
\]

\columnbreak
\end{multicols}

\begin{multicols}{2}
\noindent
para[4], ${\rm II_{3-1}}\colon H_{{\rm NY}}^{A_5}$
\[
\begin{xy}
(0,9) *+[Fo]{2}="A";
(0,18) *+[Fo]{2}="L";
 (-10,0)  *+[Fo]{B}="B";
(10,0) *+[Fo]{B}="D" ;
(0,-9) *+[Fo]{1}="H" ;
(-8,22) *+[Fo]{1}="J";
(8,22) *+[Fo]{1}="K";
\ar@{-}"L";"A";
\ar@{-}"L";"K";
\ar@{-}"A";"B";
\ar@{-}"A";"D";
\ar@{-}"D";"H";
\ar@{-}"B";"H";
\ar@{-}"L";"J";
\end{xy}
\hspace{4mm}
\begin{matrix}
\begin{pmatrix}
-1 & 0 & -1 & -1
\\
0 & 1 & 1 & 3 
\\
0 & 0 & -1 & 0
\\
0 & 0 & 0 & 1 
\end{pmatrix}
\end{matrix}
\]

\columnbreak

\noindent
${\rm IV^{*}-II_{3}}\colon H_{{\rm NY}}^{A_4}$
\[
\begin{xy}
(0,9) *+[Fo]{3}="A";
(0,16) *+[Fo]{2}="L";
 (-8,-3)  *+[Fo]{B}="B";
(8,-3) *+[Fo]{B}="D" ;
(0,-9) *+[Fo]{1}="H" ;
(-8,4) *+[Fo]{1}="I"; 
(8,4) *+[Fo]{1}="J";
(0,23) *+[Fo]{1}="K";
\ar@{-}"L";"A";
\ar@{-}"L";"K";
\ar@{-}"I";"B";
\ar@{-}"J";"D";
\ar@{-}"D";"H";
\ar@{-}"B";"H";
\ar@{-}"A";"I";
\ar@{-}"A";"J";
\end{xy}
\hspace{4mm}
\begin{matrix}
\begin{pmatrix}
-1 & 0 & -1 & -1
\\
-1 & 1 & 0 & 3 
\\
1 & 0 & 0 & 0
\\
0 & 0 & 0 & 1 
\end{pmatrix}
\end{matrix}
\]

\end{multicols}

\begin{multicols}{2}
\noindent
${\rm I_{2}-I_0^{*}-0}\colon H_{\rm Ss}^{D_6}$
\[
\begin{xy}
(-10,0) *+[Fo]{1}="A";
 (0,0)  *+[Fo]{D}="B";
(10,0) *+[Fo]{2}="C";
(10,10) *+[Fo]{1}="D" ;
(10,-10) *+[Fo]{1}="I"; 
(20,0) *+[Fo]{1}="K"; 
\ar@{=} "A";"B"
\ar@{-} "B";"C";
\ar@{-} "C";"D";
\ar@{-} "C";"I";
\ar@{-} "C";"K";
\end{xy}
\hspace{4mm}
\begin{matrix}
\begin{pmatrix}
-1 & 0 & 0 & 0
\\
0 & 1 & 0 & 2
\\
0 & 0 & -1 & 0 
\\
0 & 0 & 0 & 1 
\end{pmatrix}
\end{matrix}
\]
\columnbreak

\noindent
${\rm I_{2}-I_1^{*}-0}\colon H_{\rm Ss}^{D_5}$
\[
\begin{xy}
(-10,0) *+[Fo]{1}="A";
 (0,0)  *+[Fo]{D}="B";
(10,0) *+[Fo]{2}="C";
(10,10) *+[Fo]{1}="D" ;
(20,0) *+[Fo]{2}="M" ;
(28,5) *+[Fo]{1}="I"; 
(28,-5) *+[Fo]{1}="K"; 
\ar@{=} "A";"B"
\ar@{-} "B";"C";
\ar@{-} "C";"D";
\ar@{-} "C";"M";
\ar@{-}"M";"I";
\ar@{-}"M";"K";
\end{xy}
\hspace{4mm}
\begin{matrix}
\begin{pmatrix}
-1 & 0 & -1 & 0
\\
0 & 1 & 0 & 2
\\
0 & 0 & -1 & 0 
\\
0 & 0 & 0 & 1 
\end{pmatrix}
\end{matrix}
\]
\end{multicols}

\begin{multicols}{2}
\noindent
${\rm I_{2}-I_2^{*}-0}\colon H_{\rm Ss}^{D_4}$
\[
\begin{xy}
(-10,0) *+[Fo]{1}="A";
 (0,0)  *+[Fo]{D}="B";
(10,0) *+[Fo]{2}="C";
(10,10) *+[Fo]{1}="D" ;
(20,0) *+[Fo]{2}="L" ;
(30,0) *+[Fo]{2}="M" ;
(38,5) *+[Fo]{1}="I"; 
(38,-5) *+[Fo]{1}="K"; 
\ar@{=} "A";"B"
\ar@{-} "B";"C";
\ar@{-} "C";"D";
\ar@{-} "C";"L";
\ar@{-} "L";"M";
\ar@{-}"M";"I";
\ar@{-}"M";"K";
\end{xy}
\hspace{4mm}
\begin{matrix}
\begin{pmatrix}
-1 & 0 & -2 & 0
\\
0 & 1 & 0 & 2
\\
0 & 0 & -1 & 0 
\\
0 & 0 & 0 & 1 
\end{pmatrix}
\end{matrix}
\]
\columnbreak

\end{multicols}

\noindent
${\rm I_{2}-I_3^{*}-0}\colon H_{\rm KSs}^{\frac{3}{2}+2}$
\begin{multicols}{3}
\[
\begin{xy}
(-10,0) *+[Fo]{1}="A";
 (0,0)  *+[Fo]{D}="B";
(10,0) *+[Fo]{2}="C";
(10,10) *+[Fo]{1}="D" ;
(20,0) *+[Fo]{2}="L" ;
(30,0) *+[Fo]{2}="H" ;
(40,0) *+[Fo]{2}="M" ;
(48,5) *+[Fo]{1}="I"; 
(48,-5) *+[Fo]{1}="K"; 
\ar@{=} "A";"B"
\ar@{-} "B";"C";
\ar@{-} "C";"D";
\ar@{-} "C";"L";
\ar@{-} "L";"H";
\ar@{-} "H";"M";
\ar@{-}"M";"I";
\ar@{-}"M";"K";
\end{xy}
\hspace{4mm}
\begin{matrix}
\begin{pmatrix}
-1 & 0 & -3 & 0
\\
0 & 1 & 0 & 2
\\
0 & 0 & -1 & 0 
\\
0 & 0 & 0 & 1 
\end{pmatrix}
\end{matrix}
\]
\end{multicols}

\noindent
${\rm I_{2}-I_4^{*}-0}\colon H_{\rm KSs}^{\frac{4}{3}+2}$
\begin{multicols}{3}
\[
\begin{xy}
(-10,0) *+[Fo]{1}="A";
 (0,0)  *+[Fo]{D}="B";
(10,0) *+[Fo]{2}="C";
(10,10) *+[Fo]{1}="D" ;
(20,0) *+[Fo]{2}="L" ;
(30,0) *+[Fo]{2}="H" ;
(40,0) *+[Fo]{2}="N" ;
(50,0) *+[Fo]{2}="M" ;
(58,5) *+[Fo]{1}="I"; 
(58,-5) *+[Fo]{1}="K"; 
\ar@{=} "A";"B"
\ar@{-} "B";"C";
\ar@{-} "C";"D";
\ar@{-} "C";"L";
\ar@{-} "L";"H";
\ar@{-} "H";"N";
\ar@{-} "N";"M";
\ar@{-}"M";"I";
\ar@{-}"M";"K";
\end{xy}
\hspace{4mm}
\begin{matrix}
\begin{pmatrix}
-1 & 0 & -4 & 0
\\
0 & 1 & 0 & 2
\\
0 & 0 & -1 & 0 
\\
0 & 0 & 0 & 1 
\end{pmatrix}
\end{matrix}
\]
\end{multicols}

\noindent
${\rm I_{2}-I_5^{*}-0}\colon H_{\rm KSs}^{\frac{5}{4}+2}$
\begin{multicols}{3}
\[
\begin{xy}
(-10,0) *+[Fo]{1}="A";
 (0,0)  *+[Fo]{D}="B";
(10,0) *+[Fo]{2}="C";
(10,10) *+[Fo]{1}="D" ;
(20,0) *+[Fo]{2}="L" ;
(30,0) *+[Fo]{2}="H" ;
(40,0) *+[Fo]{2}="N" ;
(50,0) *+[Fo]{2}="O" ;
(60,0) *+[Fo]{2}="M" ;
(68,5) *+[Fo]{1}="I"; 
(68,-5) *+[Fo]{1}="K"; 
\ar@{=} "A";"B"
\ar@{-} "B";"C";
\ar@{-} "C";"D";
\ar@{-} "C";"L";
\ar@{-} "L";"H";
\ar@{-} "H";"N";
\ar@{-} "N";"O";
\ar@{-} "O";"M";
\ar@{-}"M";"I";
\ar@{-}"M";"K";
\end{xy}
\hspace{4mm}
\begin{matrix}
\begin{pmatrix}
-1 & 0 & -5 & 0
\\
0 & 1 & 0 & 2
\\
0 & 0 & -1 & 0 
\\
0 & 0 & 0 & 1 
\end{pmatrix}
\end{matrix}
\]
\end{multicols}

\noindent
${\rm I_{2}-I_6^{*}-0}\colon H_{\rm KSs}^{\frac{3}{2}+\frac{5}{4}}$
\begin{multicols}{3}
\[
\begin{xy}
(-10,0) *+[Fo]{1}="A";
 (0,0)  *+[Fo]{D}="B";
(10,0) *+[Fo]{2}="C";
(10,10) *+[Fo]{1}="D" ;
(20,0) *+[Fo]{2}="L" ;
(30,0) *+[Fo]{2}="H" ;
(40,0) *+[Fo]{2}="N" ;
(50,0) *+[Fo]{2}="O" ;
(60,0) *+[Fo]{2}="P" ;
(70,0) *+[Fo]{2}="M" ;
(78,5) *+[Fo]{1}="I"; 
(78,-5) *+[Fo]{1}="K"; 
\ar@{=} "A";"B"
\ar@{-} "B";"C";
\ar@{-} "C";"D";
\ar@{-} "C";"L";
\ar@{-} "L";"H";
\ar@{-} "H";"N";
\ar@{-} "N";"O";
\ar@{-} "O";"P";
\ar@{-} "P";"M";
\ar@{-}"M";"I";
\ar@{-}"M";"K";
\end{xy}
\hspace{4mm}
\begin{matrix}
\begin{pmatrix}
-1 & 0 & -6 & 0
\\
0 & 1 & 0 & 2
\\
0 & 0 & -1 & 0 
\\
0 & 0 & 0 & 1 
\end{pmatrix}
\end{matrix}
\]
\end{multicols}

\begin{multicols}{2}
\noindent
${\rm I_{0}-I_{0}^{*}-1}\colon H^{\rm Mat}_{\rm VI}$
\[
\begin{xy}
(-10,0) *+[Fo]{A}="A";
 (0,0)  *+[Fo]{B}="B";
(10,0) *+[Fo]{2}="C";
(10,10) *+[Fo]{1}="D" ;
(20,0) *+[Fo]{1}="H" ;
(10,-10) *+[Fo]{1}="I"; 
\ar@{-} "A";"B"
\ar@{-} "B";"C";
\ar@{-} "C";"D";
\ar@{-} "C";"H";
\ar@{-}"C";"I";
\end{xy}
\begin{matrix}
\hspace{3mm}
\begin{pmatrix}
1 & 0 & 0 & 0
\\
0 & -1 & 0 & 0 
\\
0 & 0 & 1 & 0 
\\
0 & 0 & 0 & -1 
\end{pmatrix}
\end{matrix}
\]

\columnbreak

\noindent
${\rm I_{0}-I_{1}^{*}-1}\colon H^{\rm Mat}_{\rm V}$
\[
\begin{xy}
(-10,0) *+[Fo]{A}="A";
 (0,0)  *+[Fo]{B}="B";
(10,0) *+[Fo]{2}="C";
(10,10) *+[Fo]{1}="D" ;
(20,0) *+[Fo]{2}="H" ;
(20,10) *+[Fo]{1}="I"; 
(30,0) *+[Fo]{1}="J";
\ar@{-} "A";"B"
\ar@{-} "B";"C";
\ar@{-} "C";"D";
\ar@{-} "C";"H";
\ar@{-}"H";"I";
\ar@{-} "H";"J";
\end{xy}
\hspace{4mm}
\begin{pmatrix}
1 & 0 & 0 & 0
\\
0 & -1 & 0 & -1 
\\
0 & 0 & 1 & 0 
\\
0 & 0 & 0 & -1 
\end{pmatrix}
\]

\end{multicols}

\begin{multicols}{2}
\noindent
${\rm I_{0}-I_{2}^{*}-1}\colon H^{\rm Mat}_{{\rm III}(D_{6})}$
\[
\begin{xy}
(-10,0) *+[Fo]{A}="A";
 (0,0)  *+[Fo]{B}="B";
(10,0) *+[Fo]{2}="C";
(10,10) *+[Fo]{1}="D" ;
(20,0)  *+[Fo]{2}="E" ;
(30,0) *+[Fo]{2}="H" ;
(30,10) *+[Fo]{1}="I"; 
(40,0) *+[Fo]{1}="J";
\ar@{-} "A";"B"
\ar@{-} "B";"C";
\ar@{-} "C";"D";
\ar@{-} "C";"E";
\ar@{-} "E";"H";
\ar@{-}"H";"I";
\ar@{-} "H";"J";
\end{xy}
\hspace{4mm}
\begin{pmatrix}
1 & 0 & 0 & 0
\\
0 & -1 & 0 & -2 
\\
0 & 0 & 1 & 0 
\\
0 & 0 & 0 & -1 
\end{pmatrix}
\]

\columnbreak

\end{multicols}

\begin{multicols}{2}
\noindent
${\rm I_{0}-I_{3}^{*}-1}\colon H^{\rm Mat}_{{\rm III}(D_{7})}$
\[
\begin{xy}
(-10,0) *+[Fo]{A}="A";
 (0,0)  *+[Fo]{B}="B";
(10,0) *+[Fo]{2}="C";
(10,10) *+[Fo]{1}="D" ;
(20,0)  *+[Fo]{2}="E" ;
(30,0) *+[Fo]{2}="F";
(40,0) *+[Fo]{2}="H" ;
(40,10) *+[Fo]{1}="I"; 
(50,0) *+[Fo]{1}="J";
\ar@{-} "A";"B"
\ar@{-} "B";"C";
\ar@{-} "C";"D";
\ar@{-} "C";"E";
\ar@{-} "E";"F";
\ar@{-}"F";"H";
\ar@{-}"H";"I";
\ar@{-} "H";"J";
\end{xy}
\hspace{4mm}
\begin{pmatrix}
1 & 0 & 0 & 0
\\
0 & -1 & 0 & -3 
\\
0 & 0 & 1 & 0 
\\
0 & 0 & 0 & -1 
\end{pmatrix}
\]

\columnbreak

\end{multicols}

\begin{multicols}{2}
\noindent
${\rm I_{0}-I_{4}^{*}-1}\colon H^{\rm Mat}_{{\rm III}(D_{8})}$
\[
\begin{xy}
(-10,0) *+[Fo]{A}="A";
 (0,0)  *+[Fo]{B}="B";
(10,0) *+[Fo]{2}="C";
(10,10) *+[Fo]{1}="D" ;
(20,0)  *+[Fo]{2}="E" ;
(30,0) *+[Fo]{2}="F";
(40,0) *+[Fo]{2}="G" ;
(50,0) *+[Fo]{2}="H" ;
(50,10) *+[Fo]{1}="I"; 
(60,0) *+[Fo]{1}="J";
\ar@{-} "A";"B"
\ar@{-} "B";"C";
\ar@{-} "C";"D";
\ar@{-} "C";"E";
\ar@{-} "E";"F";
\ar@{-}"F";"G";
\ar@{-} "G";"H";
\ar@{-}"H";"I";
\ar@{-} "H";"J";
\end{xy}
\hspace{4mm}
\begin{pmatrix}
1 & 0 & 0 & 0
\\
0 & -1 & 0 & -4 
\\
0 & 0 & 1 & 0 
\\
0 & 0 & 0 & -1 
\end{pmatrix}
\]

\columnbreak

\end{multicols}

\noindent
${\rm I_{0}-IV^{*}-1}\colon H^{\rm Mat}_{\rm  IV}$
\begin{multicols}{2}
\[
\begin{xy}
\ar@{-} (-10,0) *+[Fo]{A};
 (0,0)  *+[Fo]{B}="B"\ar@{-} "B";
(10,0) *+[Fo]{2}="H" \ar@{-} "H";
(20,0)  *+[Fo]{3}="C" \ar@{-} "C";
(20,10) *+[Fo]{2}="A" \ar@{-} "A";
(20,20) *+[Fo]{1}="G" \ar@{-} "G";
(30,0) *+[Fo]{2}="D" ;
(40,0) *+[Fo]{1}="E" ;
\ar@{-} "C";"D";
\end{xy}
\hspace{4mm}
\begin{pmatrix}
-1 & 0 & -1 & 0
\\
0 & 1 & 0 & 0 
\\
1 & 0 & 0 & 0 
\\
0 & 0 & 0 & 1 
\end{pmatrix}
\]	

\columnbreak

\end{multicols}

\begin{multicols}{2}
\noindent
${\rm I_{0}-III^{*}-1}\colon H^{\rm Mat}_{\rm II}$
\[
\begin{xy}
\ar@{-} (-10,0) *+[Fo]{A};
 (0,0)  *+[Fo]{B}="H"\ar@{-} "H";
(10,0) *+[Fo]{2}="B" \ar@{-} "B";
(20,0)  *+[Fo]{3}="C" \ar@{-} "C";
(30,0) *+[Fo]{4}="D" \ar@{-} "D";
(30,10) *+[Fo]{2}="A" \ar@{-} "A";
(40,0) *+[Fo]{3}="E";
(50,0) *+[Fo]{2}="F" \ar@{-} "F";
(60,0) *+[Fo]{1}="G";
\ar@{-} "D";"E";
\end{xy}
\]

\columnbreak

	\[
	\begin{pmatrix}
		0 & 0 & -1 & 0
		\\
		0 & 1 & 0 & 0 
		\\
		1 & 0 & 0 & 0 
		\\
		0 & 0 & 0 & 1 
	\end{pmatrix}
	\]

\end{multicols}

\begin{multicols}{2}
\noindent
${\rm I_{0}-II^{*}-1}\colon H^{\rm Mat}_{\rm I}$
\[
\begin{xy}
\ar@{-} (-10,0) *+[Fo]{A};
 (0,0)  *+[Fo]{B}="H"\ar@{-} "H";
(10,0) *+[Fo]{2}="B" \ar@{-} "B";
(20,0)  *+[Fo]{3}="C" \ar@{-} "C";
(30,0) *+[Fo]{4}="D" \ar@{-} "D";
(40,0) *+[Fo]{5}="E" \ar@{-} "E";
(50,0) *+[Fo]{6}="F" \ar@{-} "F";
(50,10) *+[Fo]{3}="A" \ar@{-} "A";
(60,0) *+[Fo]{4}="G";
(70,0) *+[Fo]{2};
\ar@{-} "F";"G";
\end{xy}
\]

\columnbreak

	\[
	\begin{pmatrix}
	0 & 0 & -1 & 0
	\\
	0 & 1 & 0 & 0 
	\\
	1 & 0 & 1 & 0 
	\\
	0 & 0 & 0 & 1 
	\end{pmatrix}
	\]
\end{multicols}

%


\def\cydot{\leavevmode\raise.4ex\hbox{.}} \def\cprime{$'$}

\end{document}